\documentclass[12pt,leqno]{article}
\usepackage[a4paper, margin=1in]{geometry}
\usepackage{amsmath, amsthm, amsfonts, amssymb}
\usepackage{mathrsfs}
\usepackage{enumitem}
\usepackage{bm}
\usepackage[misc]{ifsym}
\usepackage{hyperref}

\newtheorem{thm}{Theorem}[section]
\newtheorem{cor}[thm]{Corollary}
\newtheorem{lem}[thm]{Lemma}
\newtheorem{prp}[thm]{Proposition}

\newtheorem{exa}[thm]{Example}
\theoremstyle{definition}
\newtheorem{rem}[thm]{Remark}
\newtheorem{defn}{Definition}[section]

\numberwithin{equation}{section} 

\allowdisplaybreaks

\newcommand{\scr}[1]{\mathscr #1}

\def\eins{\boldsymbol 1}
\def\RR{\mathbb R} 
\def\NN{\mathbb N} 
\def\WW{\mathbb W}

\def\EE{\mathbb E}

\def\E{\scr E}
\def\D{\scr D}
\def\B{\scr B}
\def\F{\scr F}
\def\P{\scr P}
\def\N{\scr N}

\def\lam{\lambda}
\def\Lam{\Lambda}

\def\na{\nabla}
\def\la{\langle}
\def\ra{\rangle}

\def\d{\textnormal {d}}

\def\id{\textnormal{id}}
\def\supp{\textnormal{supp}}

\def\tr{\textnormal {T}}
\def\tra{\textnormal {tr}}

\def\euc{\textnormal{euc}}
\def\hes{\na^{2,\euc}}
\def\ac{\textnormal {ac}}
\def\div{\textnormal{div}}

\title{The Ornstein--Uhlenbeck process on $\scr P_2$ \\with a volatility operator}
\author{Martin Grothaus$^{(a)}$ and Simon Wittmann$^{(b)}$\footnote{ Corresponding author\quad \Letter \;\textit{ simon.wittmann@pwr.edu.pl}}}

\begin{document}

\maketitle

\begin{abstract}
	We analyze diffusion processes ${(\bm \mu_t)}_{t\geq 0}$ on the $2$-Wasserstein space $\scr P_2$ over $\RR^d$,
	which can be classified as Ornstein--Uhlenbeck processes by means of their Fukushima decomposition.
	Characteristically for this class,
	$	|\bm \mu_t|_2^2-|\bm \mu_0|_2^2-2ct+2\int_0 ^t|\bm \mu_s|_2^2\,\d s,\, t\geq 0,$
	is a martingale, where $c\in(0,\infty)$ is a constant and $|\bm \mu_t|_2:=(\int_{\RR^d}x^\tr x\bm\mu_t(\d x ))^{1/2}$
	denotes the second moment. 
	Making a particular choice from this class, we obtain a process ${(\bm \mu_t)}_{t\geq 0}$ such that all moments exist and
	${\bm\mu_t(\theta)}-{\bm\mu_0(\theta)}-\int_0 ^t\bm\mu_s\big(\Delta \theta-\la \id_{\RR^d},\na \theta\ra_\euc\big)\d s$, $t\geq 0$,
	is a martingale for each multivariate polynomial $\theta:\RR^d\to\RR$.
\end{abstract}

\noindent 2020 Mathematics subject classification: 60J25, 60J60, 60J46.\\
\noindent Keywords: Diffusion process, Wasserstein space, Dirichlet form, Ornstein--Uhlenbeck process, Gaussian measure

\footnotesize{\begin{itemize}\item[(a)] Department of Mathematics,  RPTU University Kaiserslautern-Landau,   67663 Kaiserslautern,  Germany
\item[(b)] Faculty of Mathematics, Wroc\l aw University of Science and Technology,   50-370 Wroc\l aw,  Poland\end{itemize}}


\section{Introduction}\label{sec:intro}

\underline{Motivation for an Ornstein--Uhlenbeck operator on $\P_2$}

Our aim is the construction and analysis of a probability measure valued diffusion process ${(\bm \mu_t)}_{t\geq 0}$, whose generator $L$ can be identified as the generator  of an Ornstein--Uhlenbeck type semigroup.
As there is no general consent about what the notion of an  Ornstein--Uhlenbeck semigroup acting on functions $u:\P_2\to\RR$ would be,  our terminology needs explanation.
The most striking argument worked out in this article is a duality type connection between this newly introduced generator $L$  and the generator
of the transition semigroup of the classical Ornstein--Uhlenbeck process ${(X_t)}_{t\geq 0}$ on $\RR^d$. The latter satisfies
\begin{equation*}
	\d X_t=-X_t\d t+ \sqrt {2}\, \d W_t,\qquad t\geq 0,
\end{equation*}
for a $d$-dimensional standard Brownian motion ${(W_t)}_{t\geq 0}$.
 For the simple evaluation map $u:\mu\mapsto \int_{\RR^d}\theta\d\mu $ at a function $\theta\in C^2(\RR^d)$ such that all second order partial derivatives of $\theta$ are dominated by polynomials, we have 
\begin{equation}\label{eq:LOU}
	Lu(\mu)=\int_{\RR^d}\Delta \theta (x) -\la x,\na \theta (x)\ra_\euc\,\mu(\d x),
\end{equation}
i.e.~the generator $L$ sends the evaluation map at $\theta$ onto the evaluation map at the classical Ornstein--Uhlenbeck  operator applied to $\theta$.
The equality in \eqref{eq:LOU} has to be read in a.e.-sense in the argument $\mu$ w.r.t.~a suitable reference measure on $\P_2$, which is the invariant measure of ${(\bm \mu_t)}_{t\geq 0}$. 
In fact, $L$ is the generator of a measure-valued diffusion, which is one significant example in a general class of diffusion processes  discussed in this article.
Let us briefly point out our motivation, before we shed more light on this novel class of Ornstein--Uhlenbeck processes
on $\P_2$.

For at least 20 years, there has been an extensive discussion on the question of how a canonical diffusion process on the Wasserstein space can be defined (\cite{Sturm,Shao,Kon14,CG,Sch,Del,St24}).
A point in which most authors in this field agree: such a diffusion must be geometrically related to Otto's gradient (see \cite{Otto} and also \cite{AKR96, AKR98, R98, Sch2, BRW21, RW22, RWW24}
in addition to the references mentioned above).
A primary reason for the controversy about a canonical diffusion
- apart from the lack of a Lebesgue measure - is connected to an observation about the Laplace operator on the Wasserstein space, formulated in \cite[Rem~5.6]{Gi12}:
the Wasserstein Laplace cannot be defined in a strict everywhere-sense for a sufficiently large domain of functions.
Nevertheless, the first list of references cited above in this paragraph holds a selection of highly relevant reference measures as valuable substitutes for a volume measure, as well as related
diffusion processes whose generators  coincide with the Wasserstein Laplace or a modification of it.
The  nature of the underlying reference measure play a major role in this aspect, as well as which Wasserstein space is considered (i.e.~over which base). Most of the promising candidates  use the setting of a closed Riemannian manifold
or a one-dimensional space and the reference measures disregard the absolutely continuous elements of the Wasserstein space.

In this context, we combine two thoughts to explain the motivation  for this article:
\begin{itemize}
	\item For a lot of physical problems, the observable states of a system are  probability density functions on the Euclidean space $\RR^d$. 
	\item Gaussian measures emerge from various physical phenomena and have had a big impact on the theory of  SDEs on infinite dimensional vector spaces. A viable notion of Gaussian measures,
	or  associated Ornstein--Uhlenbeck processes, can be expected to go a long way in stochastic analysis on infinite dimensional manifolds.
\end{itemize}
Therefore:
\begin{itemize}
	\item  We construct a non-degenerate diffusion on the Wasserstein space $\P_2$ over the Euclidean space $\RR^d$.
	\item The generator is an Ornstein--Uhlenbeck operator (see Theorem \ref{thm:thmL} and Proposition \ref{prp:diffO}), symmetric w.r.t.~to a Gaussian reference measure on $\P_2$
	which is defined by its covariance operator. 
	\item That reference measure assigns probability $1$ to the subset of absolutely continuous elements of $\P_2$ for a reasonable large class of those covariance operators (Lemma \ref{lem:pac}).
\end{itemize}

A priori, it is not clear at all why  an extension of the domain of the operator $L$ in \eqref{eq:LOU} exists such that $(L,\D(L))$ is the generator of a semigroup associated with a Markov process
on $\P_2$. For the existence we use the quasi-regularity of a Dirichlet form on $\P_2$, whose domain is a superset of all bounded differentiable functions $u:\P_2\to\RR$
(by a general criterion proven in \cite{RWW24}). 
Our new strategy how to obtain a characterization for the generator is to modify Otto's structure slightly, but in a way which relates naturally to the problem and does not degenerate the diffusion.
We manipulate the inner products ${\la\cdot,\cdot\ra_{L^2(\RR^d,\mu;\RR^d)}}$ by a field of strictly positive, symmetric bounded operators $K_\mu$ on $L^2(\RR^d,\mu;\RR^d)$ over $\mu\in\P_2$,
which arises from an underlying Gaussian measure (see Sections \ref{sec:B} \& \ref{sec:SF}).
To understand in more detail why it is legitimate to claim the terminology of Ornstein--Uhlenbeck for the class of measure-valued diffusion processes analyzed in this article and hence for their generators,
we recall the classical situation on Euclidean space before we draw an analogy.

\underline{The Ornstein--Uhlenbeck process on $\mathbb R^d$ with a volatility matrix}

We fix $d\in\NN$ and a symmetric, strictly positive definite (non-random) matrix  $\Sigma={(\Sigma_{k,l})}_{k,l\in\{1,\dots,d\}}\in\RR^{d\times d}$.
Let ${(X_t)}_{t\geq 0}$ be a $\RR^d$-valued stochastic process satisfying
\begin{equation}\label{eq:ex1}
	 X_t=-\int_0^tX_s\d s+ \sqrt {2\Sigma}\, W_t,\qquad t\geq 0,
\end{equation}
where ${( W_t)}_{t\geq 0}$
denotes a standard Brownian Motion on $\RR^d$ and $\sqrt {2\Sigma}$ the symmetric positive matrix with $(\sqrt {2\Sigma})^2=2\Sigma$.
Equation \eqref{eq:ex1} describes a  mean-zero  Ornstein--Uhlenbeck process whose volatility coefficient is the (constant, non-random) matrix $\sqrt {2\Sigma}$.

The component processes ${\{{( W_t^{\Sigma,k})}_{t\geq 0}\}}_{k=1,\dots,d}$ of ${( W^\Sigma_t)}_{t\geq 0}:={(\sqrt{\Sigma} W_t)}_{t\geq 0}$ are martingales with quadratic covariation processes
\begin{equation*}
	{\big\la  W_\bullet ^{\Sigma,k}, W_\bullet ^{\Sigma,l}\big\ra}_t=t\Sigma_{k,l},\qquad t\geq 0,\,k,l\in\{1,\dots, d\}.
\end{equation*}
The invariant measure of \eqref{eq:ex1} is the Gaussian measure 
\begin{equation*}
	\big((2\pi)^d\det (\Sigma)\big)^{-\frac{1}{2}}\exp\big(-\tfrac{1}{2} x^\tr \Sigma^{-1}x\big)\d x
\end{equation*}
with covariance operator $\Sigma$.
 For a function $\varphi\in C^2(\RR^d)$ the multi-dimensional Itô formula yields
\begin{equation}\label{eq:intro1}
	\d(\varphi\circ X_t)=\sum_{k=1}^d\Big[\Big(\sum_{l=1}^d\Sigma_{k,l}\partial_k\partial_l \varphi(X_s)\Big)-X_s^k\partial_k\varphi(X_s)\Big]\d s
+\sqrt 2\sum_{k=1}^d \partial_k\varphi(X_s) \d  W^{\Sigma, k }_s
\end{equation}
for $t\geq 0$, where ${\{{(X_t^k)}_{t\geq 0}\}}_{k=1,\dots,d}$ denote the component processes of ${(X_t)}_{t\geq 0}$.
 In particular, writing $|x|_\euc=\sqrt{x^\tr x}$ for $x\in\RR^d$ and choosing $\varphi(x):=|x|_\euc^2$ results in
\begin{equation}\label{eq:intro2}
	|X_t|_\euc^2=2\tra(\Sigma)t-2\int_0 ^t|X_s|_\euc^2\d s+2 \sqrt 2\int_{0}^tX_s^\tr \d  W^{\Sigma}_s,\qquad t\geq 0.
\end{equation}

\underline{The Ornstein--Uhlenbeck process on the Wasserstein space with a volatility operator}

We write $\delta_x$ for the Dirac measure in a point $x\in\RR^d$. In this article, we define and analyze a diffusion ${(\bm\mu_t)}_{t\geq 0}$ on the $2$-Wasserstein space $\P_2$ over $\RR^d$,
whose drift and martingale parts are modeled after those of the Ornstein--Uhlenbeck process on $\RR^d$. In the drift part of \eqref{eq:intro1}, respectively \eqref{eq:intro2}, can be written
\begin{equation}\label{eq:driftA}
\int_0 ^t\int_{\RR^d}	\sum_{k=1}^d\Big[\Big(\sum_{l=1}^d\Sigma_{k,l}\partial_k\partial_l \varphi(x)\Big)-x_k\partial_k\varphi(x)\Big]{\delta}_{X_s}(\d x)\d s,\quad t\geq 0,
\end{equation}	
and 
\begin{equation}\label{eq:driftB}
	2\tra(\Sigma)t-2\int_0^t\int_{\RR^d}|x|_\euc^2\delta_{X_s}(\d x)\d s,\quad t\geq 0.
\end{equation}

We want to replace integration w.r.t.~${\delta_{X_s}}$ in \eqref{eq:driftA}, \eqref{eq:driftB} by integration w.r.t.~the time-dependent random measure ${\bm\mu_s}$
which are  the trajectories of a non-degenerate diffusion on $\scr P_2$ constructed in this article. In the following we describe the evident optical analogy in the Itô formulas for ${(X_t)}_{t\geq 0}$ on the one side
and ${(\bm\mu_t)}_{t\geq 0}$ on the other side. However, it shall be stressed that on a technical level it is not evident at all how the sample paths of ${(X_t)}_{t\geq 0}$ could be \lq lifted\rq\, 
to create the paths of a diffusion on $\P_2$ which is non-degenerate. 

At first, the notion of $\Sigma$ has to be adapted.
In the example of the Ornstein--Uhlenbecck process \eqref{eq:ex1}, the parameter $\Sigma$ is chosen as any symmetric, strictly positive definite matrix.
For our diffusion  ${(\bm\mu_t)}_{t\geq 0}$ on $\P_2$, the parameter to choose is not a matrix, but an operator $A$ of the following type:
\begin{itemize}
	\item $A$ is a self-adjoint, strictly positive definite operator on the Hilbert space $L^2(\RR^d,\lam;\RR^d)$ for some absolutely continuous probability measure $\lambda$.
	\item The inverse operator $A^{-1}$ has a finite trace. 
\end{itemize}
In the first of these two conditions, the particular choice of $\lambda$ is irrelevant, meaning for any such $\lambda$ we can 
find an operator $A$ as above which generates ${(\bm\mu_t)}_{t\geq 0}$ (see the discussion in Section \ref{sec:supp}). The second condition implies that 
$A^{-1}$ is the covariance operator of a (unique) centered, non-degenerate Gaussian measure $G$ on 
\begin{equation*}
T_\lam:=L^2(\RR^d,\lam;\RR^d).
\end{equation*} 
Let $K^A={(K^A_{k,l})}_{k,l\in\{1,\dots,d\}}:\RR^d\times\RR^d\to\RR^{d\times d}$ denote the integral kernel of $A^{-1}$, meaning
\begin{equation*}
	K^A(x,y)=\sum_{n=1}^\infty\alpha_ng_n(x)g_n(y)^\tr
\end{equation*}
for the sequence of eigenvalues ${(\alpha_n)}_{n\in\NN}$ of $A^{-1}$ (listed with multiplicities) with a corresponding orthonormal eigenbasis ${(g_n)}_{n\in\NN}$. Moreover, we define
\begin{align*}
	\Psi_\lam:T_\lam\ni f&\;\mapsto\; \lam\circ f^{-1}\in\P_2\\
\text{and}\qquad\qquad	\vec \Psi_\lam:T_\lam\times\RR^d\ni (f,x)&\;\mapsto\; (\Psi_\lam(f), f(x))\in\P_2\times \RR^d.
\end{align*}
There is family of coefficient functions $\{\rho^{\mu}_{k,l}:\RR^d\to\RR$ $|$ $\mu\in\P_2,\,k,l\in\{1,\dots d\}\}$ depending only on $A$, such that the drift part of 
\begin{equation}\label{eq:introPhi}
	\int_{\RR^d}\varphi\d\bm\mu_t-\int_{\RR^d}\varphi\d\bm\mu_0\qquad\text{for }  t\geq  0,\, \varphi\in C^2(\RR^d),\,\sup_{x\in\RR^d}|\partial_k\partial_l\varphi(x)|<\infty,
	\,k,l\in\{1,\dots,d\},
\end{equation}
is given as
\begin{equation}\label{eq:driftAA}
	\int_0 ^t\int_{\RR^d}	\sum_{k=1}^d\Big[\Big(\sum_{l=1}^d\rho^{\bm\mu_s}_{k,l}(x)\partial_k\partial_l \varphi(x)\Big)-x_k\partial_k\varphi(x)\Big]\bm\mu_s(\d x)\d s.
\end{equation}
The drift part of $\int_{\RR^d}|\cdot|_\euc^2 \d\bm\mu_t-\int_{\RR^d}|\cdot|_\euc^2\d\bm\mu_0$, $t\geq 0$, is a special case and given by 
\begin{equation}\label{eq:driftBB}
	2\tra(A^{-1})t-2\int_0^t\int_{\RR^d}|x|_\euc^2\d\bm\mu_s(\d x)\d s.
\end{equation}
\eqref{eq:driftAA}, \eqref{eq:driftBB} are the re-interpretations of \eqref{eq:driftA}, \eqref{eq:driftB} for a Ornstein--Uhlenbeck process on $\P_2$.
In fact, the coefficients $(\mu,x)\mapsto \rho_{k,l}^\mu(x)$ in \eqref{eq:driftAA} coincide with the Radon-Nikodym derivative of the finite, signed image measure $(G\times K^A_{k,l}(x,x)\lam(\d x))\circ\vec \Psi_\lam^{-1}$
w.r.t.~$(G\times \lam)\circ\vec \Psi_\lam^{-1}$. In the case, where the generator of ${(\bm\mu_t)}_{t\geq 0}$ satisfies \eqref{eq:LOU}, a particular choice for $A$ is made with the outcome that
$\rho_{k,l}^\mu(x)=\delta_{kl}$ (see Examples \ref{exa:rbf}, \ref{exa:rbf2} \& \ref{exa:rbf3} below).

Concerning the quadratic variation of the martingale part, which for \eqref{eq:intro1} is given as
\begin{equation}\label{eq:QVrd}
	2\int_0^t\int_{\RR^d}\big|\sqrt{\Sigma}\na\varphi(x)\big|^2_\euc \delta_{X_s}(\d x)\d s,\quad t\geq 0,
\end{equation}
a norm-square dependence on the gradient $\na\varphi$ also occurs for the semi-martingale in \eqref{eq:introPhi}.
There exists a family of Hilbert spaces $\{H_\mu:\mu\in \P_2\}$ such that $L^2(\RR^d,\mu;\RR^d)\hookrightarrow H_\mu$ for $\mu\in\P_2$ are densely, continuously embedded and the quadratic variation of the martingale part of \eqref{eq:introPhi} reads
\begin{equation}\label{eq:QV}
	2\int_0^t\big\|\na\varphi\|_{H_{\bm \mu_s}}^2 \d s,\quad t\geq 0
\end{equation}
(see Section \ref{sec:SF} for details).
The family $\{H_\mu:\mu\in \P_2\}$ depends only on $A$ and is characterized by the identity
\begin{equation}\label{eq:QV2}
	\int_{\P_2}{\big\|\eta(\mu,\cdot)\big\|}_{H_\mu}^2(G\circ\Psi_\lam^{-1})(\d\mu)=\int_{T_\lam} \Big\| A^{-\frac{1}{2}}\eta\big(\Psi_\lam(f),f(\cdot)\big )\Big\|_{T_\lam}^2G(\d f),
	\quad\eta\in C_b(\P_2\times\RR^d,\RR^d).
\end{equation}
In view of \eqref{eq:QV2} the state-dependent quadratic forms $\|\na\,\cdot\,\|_{H_{\bm\mu_t}}^2$ defining the quadratic variation \eqref{eq:QV} are suitable re-interpretations for the state-dependent quadratic forms
$\int_{\RR^d}|\sqrt{\Sigma}\na\,\cdot\,|^2_\euc\d\delta_{X_t}$  for $t\geq 0$, in order to clear the degeneracy in \eqref{eq:QVrd}.
This completes the picture of the analogy between ${(\bm\mu_t)}_{t\geq 0}$ and ${(X_t)}_{t\geq 0}$.

\underline{Pioneering references and an outline of the article}

The invariant measure of  our process ${(\bm\mu_t)}_{t\geq 0}$ is the image of $G$ under $\Psi_\lam$. This measure has first been introduced in \cite{RW22}, where
it is referred to as a Gaussian measure on $\P_2$.  In Lemma \ref{lem:pac} below, we provide a simple criterion on how to choose  $G$ in order to obtain a  measure on $\P_2$ which assigns the probability $1$ to the subset of absolutely continuous elements  in $\P_2$.
The Dirichlet forms and the associated stochastic processes analyzed in \cite{RW22}, \cite{RWW24} differ in an essential detail from the ones  considered in this article,
but both of these references (\cite{RW22} thanks to the Gaussian measure on $\P_2$ and the idea for defining a gradient-type Dirichlet form on that state space,
\cite{RWW24} thanks to the quasi-regularity result) provide the foundational preliminaries needed for the results in this article.  In our case, the covariance operator $A^{-1}$ of $G$ appears in the square-field operator (compare Lem.~\ref{lem:closable} below with \cite[Thm.~3.2]{RW22}). This difference manifests itself in the fact that $A^{-1}$ appears in the expression for the  quadratic variation \eqref{eq:QV}, \eqref{eq:QV2}.
Another manifestation  of the impact, which our different choice of  square-field operator has, is the   explicit characterization of the generator
(compare Thm.~\ref{thm:thmL} and Prop.~\ref{prp:diffO}  below with \cite[Def.~4.1 \& Thm.~4.1]{RW22}). This allows us to identify the drift part in the Fukushima decomposition (see Thm.~\ref{thm:Fuku} and Cor.~\ref{cor:secmomProcess} below) as in \eqref{eq:driftAA} respectively \eqref{eq:driftBB}.

The outline of this article is as follows. Section \ref{sec:Dform} contains a discussion of the relevant symmetric Dirichlet form $(\E^{G,\lam},\D(\E^{G,\lam}))$ in $L^2(\scr P_2, G\circ\Psi_\lam^{-1})$ and the support properties of the reference measure
$G\circ\Psi_\lam^{-1}$.
In Section \ref{sec:generator}  an explicit description for the generator $L_{G,\lam}$ of $\E^{G,\lam}$ is derived.
The associated Markov process ${(\bm\mu_t)}_{t\geq 0}$ on $\scr P_2$ is characterized in Section \ref{sec:Fuku} and provides  martingale solutions
for $L_{G,\lam}$. Hence, we obtain the representation described above.

The main results are:
\begin{itemize}
	\item The characterization of $L_{G,\lam}$ as a second order differential operator on an explicit, measure determining class of test functions, see Theorem \ref{thm:thmL} and Proposition \ref{prp:diffO}.
	\item The decomposition of \eqref{eq:introPhi} into its drift-  and martingale-part, see Theorem \ref{thm:Fuku} and Corollary \ref{cor:secmomProcess}.
\end{itemize}

In this text, $d\in\NN$ is fixed, $\id_{\RR^d}:\RR^d\to\RR^d$ denotes the identity map, $\la x,y\ra_\euc:=x^\tr y$ the standard inner product and $|x|_\euc:=\sqrt{x^\tr x}$
the $2$-norm for $x,y\in\RR^d$.

\section{The Dirichlet form and its reference measure}\label{sec:Dform}

The process ${(\bm\mu_t)}_{t\geq 0}$ described in the introduction is associated with a gradient-type Dirichlet form on the $2$-Wasserstein space $\P_2$. 
 The Dirichlet form is the starting point of our construction and introduced in
Section \ref{sec:B}. Further properties of its reference measure on $\P_2$, which are relevant to the remainder of this article, are discussed in Section \ref{sec:supp}.

\subsection{The Dirichlet form \texorpdfstring{$\E^{G,\lam}$}{}}\label{sec:B}
For a Radon measure $\mu$ on $\RR^d$, equipped with the Borel $\sigma$-algebra, we set
\begin{equation*}
	T_\mu:=L^2(\RR^d,\mu;\RR^d),
\end{equation*}
the set of $\mu$-equivalence classes of measurable functions $f:\RR^d\to\RR^d$, $\int_{\RR^d}|f|_\euc^2\d\mu<\infty$. $T_\mu$ is a separable Hilbert space with scalar product
\begin{equation*}
	{\la f,g\ra}_{T_\mu}:=\int_{\RR^d} {\la f(x),g(x)\ra}_\euc\,\mu(\d x),\qquad f,g\in T_\mu.
\end{equation*}
When introducing a suitable space of measures and interpreting it as an infinite-dimensional Riemannian manifold, $T_\mu$ plays the role of the tangential space at a measure $\mu$
 regarding the intrinsic derivative operator, see \cite{AKR96, AKR98, R98}. This structure is also used in the setting of the $2$-Wasserstein space in order to identify the solutions 
 of some parabolic PDEs as gradient flows, see \cite{Otto, AGS05}.

The $2$-Wasserstein space $\scr P_2:=\{\mu:$ $\mu$ is a probability measure on $\RR^d$ and  $\int_{\RR^d}|\cdot|_\euc^2\d\mu<\infty\}$ is endowed with the usual (complete and separable)  $2$-Wasserstein metric
\begin{equation*}
	\WW_2(\mu,\nu):=  \inf_{\pi\in\textnormal{Coupl}(\mu,\nu)}\bigg(\int_{\RR^d\times\RR^d}|x-y|_\euc^2\pi(\d x,\d y)\bigg)^{\frac 1 2 },
\end{equation*}
where the infimum is taken over all couplings $\pi$ of $\mu$ and $\nu$.

\begin{defn}\label{def:intrinsic}
\begin{itemize}[leftmargin=*]
	\item[] a) An \textit{intrinsically differentiable} function $u:\scr P_2\to\RR$ is a continuous map such that
	$$L^2(\RR^d,\mu;\RR^d)\ni f\mapsto D_f u(\mu):=\lim_{\varepsilon\to 0}\frac{u\big(\mu\circ(\id_{\RR^d}+\varepsilon f)^{-1}\big)-u(\mu)}\varepsilon $$
	is a bounded linear functional for every $\mu\in \scr P_2$. In this case, 
	the intrinsic derivative of $u$ at $\mu$ is defined as the unique element $Du(\mu)\in T_{\mu}$ such that
	$$D_f u(\mu)={\la Du(\mu), f\ra}_{T_{\mu}},\qquad f\in T_{\mu}.$$
	\item[] b) The class $C^1(\scr P_2)$ contains all intrinsically differentiable functions $u:\scr P_2\to\RR$   which satisfy
	\begin{equation*}
	\lim_{n\to\infty}\frac{\big|u\big(\mu\circ(\id_{\RR^d}+f_n)^{-1}\big)-u(\mu)-D_{f_n} u(\mu)\big|}{\|f_n\|_{T_{\mu}}}=0,
	\qquad \mu\in\P_2,\,{(f_n)}_{n\in\NN}\subset T_\mu,\, f_n\to 0,
	\end{equation*}
	and moreover $Du(\mu)$ possesses $\mu$-versions $\widetilde {Du}(\mu)$, $\mu\in\scr P_2$, such that $$\P_2\times\RR^d\ni (\mu,x)\mapsto \widetilde {Du}(\mu)(x)\in\RR^d$$ is  continuous.
	\item[]c) The class $C^1_b(\scr P_2)$ contains all elements $u\in C^1(\scr P_2)$ such that   $$\sup_{(\mu,x)\in\P_2\times\RR^d}|u(\mu)|+|\widetilde{Du}(\mu)(x)|_{\euc}<\infty.$$ 
	\end{itemize}
\end{defn}

The probability measures on $\RR^d$ which are absolutely continuous w.r.t.~the $d$-dimensional Lebesgue measure are denoted by $\P^\ac$.

Throughout this text, we fix $\lam\in\P^\ac$ and consider the surjective contraction:
\begin{equation}\label{eq:DefPsi}
	\Psi_\lam:T_\lam\to\scr P_2,\qquad f\mapsto \lam\circ f^{-1}.
\end{equation}
We note that for $f_1,f_2\in T_\lam$,
\begin{equation*}
	\WW_2\big(\Psi_\lam(f_1),\Psi_\lam(f_2)\big)^2\leq \int_{\RR^d\times\RR^d}|x-y|_\euc^2\big(\lam\circ(f_1,f_2)^{-1}\big)(\d x,\d y)=\|f_1-f_2\|_{T_\lam}^2.
\end{equation*}
The surjectivity follows from the fact that, given any probability measure $\mu$ on $\RR^d$, there is a transport map $f:\RR^d\to\RR^d$ such that
$\lam\circ f^{-1}=\mu$ (see e.g.~\cite[Chap.~1]{V09}). In particular, 
\begin{equation*}
\int_{\RR^d}|f|_\euc^2\d\lam=\int_{\RR^d}|\cdot|_\euc^2\d\mu<\infty
\end{equation*}
 for $\mu\in\P_2$.

Let $(A,\D(A))$ be a self-adjoint, strictly positive operator on $T_\lam$ whose inverse $A^{-1}:T_\lam\to T_\lam$ is trace class. 
There is a unique centered Gaussian measure $G$ on $T_\lam$ with covariance 
\begin{equation}\label{eq:covG}
	\int_{T_\lam}{\la f_1, g\ra}_{T_\lam}{\la f_2, g\ra}_{T_\lam}G(\d g)={\big\la f_1,A^{-1}f_2\big\ra}_{T_\lam},
\end{equation}
see e.g.~\cite[Chap.~2]{B98} or \cite[Chap.~1]{DZ02}.
We call $A^{-1}$ the covariance operator of $G$.
By virtue of Lemma \ref{lem:closable} below, we can define a symmetric Dirichlet form (see \cite[Chap.~I]{BH91}, \cite[Chap.~I]{MR92}) with state space $\P_2$, endowed with the Borel $\sigma$-algebra  and the probability measure  $G\circ\Psi_\lam^{-1}$. Throughout this text, we always take $G$ as a centered Gaussian measure for simplicity. In fact, analogous results can be obtained starting with a
Gaussian measure on $T_\lam$ with (possibly non-zero) mean in $T_\lam$ and covariance operator $A^{-1}$ as above.

The measure $G\circ\Psi_\lam^{-1}$ is called a Gaussian measure on $\P_2$  in \cite{RW22}.

It is worth mentioning that the choice of $\lambda\in \P^\ac$ does not restrict the class of all Gaussian-based measures $G\circ\Psi_\lam^{-1}$ obtained in the above way, as explained in Section \ref{sec:supp} below. In the remainder of Section \ref{sec:B} we define the Dirichlet form $\E^{G,\lam}$ in $L^2(\P_2,G\circ\Psi_\lam^{-1})$.
\begin{lem}\label{lem:closable}
The symmetric, non-negative bilinear form
\begin{equation}\label{eq:preDirichlet}
	\E^{G,\lam}(u,v):=\int_{T_\lam}{\big\la  Du(\lam\circ f^{-1})\circ f,A^{-1}\big(Du(\lam\circ f^{-1})\circ f\big)\big\ra}_{T_\lam}G(\d f),\qquad u,v\in C_b^1(\scr P_2),
\end{equation}
is a pre-Dirichlet form in $L^2(\scr P_2,G\circ\Psi_\lam^{-1})$, i.e.~it has the Markov property and is closable.
\end{lem}
\begin{proof}
	We denote the $(1,2)$-Sobolev space on $T_\lam$ w.r.t.~the Gaussian measure $G$ (see \cite[Chap.~II.9]{DZ02}) by $W^{1,2}(T_\lam,G)$. Let ${(g_n)}_n$ be an eigenbasis $A^{-1}g_n=\alpha_ng_n$ with the bounded sequence of eigenvalues ${(\alpha_n)}_n\subset(0,\infty)$. Then,
\begin{equation*}
	\int_{T_\lam}{\big\la \na w(f),A^{-1}\na w(f)\big\ra}_{T_\lam}G(\d f)=\sum_{n\in\NN }\alpha_n\int_{T_\lam}{\la \na w(f),g_n\ra}_{T_\lam}^2G(\d f)
\end{equation*}
for $w\in W^{1,2}(T_\lam,G)$ and the operators $w\mapsto {\la \na w,g_n\ra}_{T_\lam}$, $n\in\NN$, are  closable in $L^2(T_\lam,G)$ (a consequence of \cite[Lem.~9.2.5]{DZ02}).
So, the minimal closed extension of the symmetric, non-negative bilinear form 
\begin{equation}\label{eq:WAform}
	(w_1,w_2)\;\mapsto\; \int_{T_\lam}{\big\la \na w_1(f),A^{-1}\na w_2(f)\big\ra}_{T_\lam}G(\d f),\qquad w_1,w_2\in W^{1,2}(T_\lam, G),
\end{equation}
in $L^2(T_\lam,G)$ exists due to \cite[Prop.~I.3.7]{MR92}. Its domain is denoted by $W^{1,2}_{A^{-1}}(T_\lam, G)$ and
we write ${\la\cdot,\cdot\ra}_{W^{1,2}_{A^{-1}}(T_\lam, G)}$ 
for the extension of the bilinear form \eqref{eq:WAform} to $W^{1,2}_{A^{-1}}(T_\lam, G)$.

As shown in \cite[Lem.~3.2]{RWW24} (also see \cite[Thm.~2.1]{BRW21}, \cite[Prop.~2.2]{RW22}) the composition $u\circ\Psi_\lam:T_\lam\to\RR$ for $u\in C_b^1(\scr P_2)$
is a Fréchet differentiable function, and moreover 
the (Fréchet-) gradient $\na(u\circ\Psi_\lam):T_\lam\to T_\lam$  satisfies
\begin{equation}\label{eq:chainR}
	\na(u\circ\Psi_\lam)(f)=Du(\Psi_\lam(f))\circ f,\qquad f\in T_\lam.
\end{equation}
In particular, $u\circ\Psi_\lam, v\circ\Psi_\lam\in W^{1,2}(T_\lam,G)\subset W_{A^{-1}}^{1,2}(T_\lam,G)$ and
\begin{align*}
\E^{G,\lam}(u,v)&=\int_{T_\lam}{\big\la \na(u\circ\Psi_\lam)(f),A^{-1}\na(u\circ\Psi_\lam)(f)\big\ra}_{T_\lam}G(\d f)
\\&={\la u\circ\Psi_\lam,v\circ\Psi_\lam\ra}_{W^{1,2}_{A^{-1}}(T_\lam, G)}
\end{align*}
for $u,v\in C_b^1(\P_2)$.
Since $\chi\circ u\in C_b^1(\P_2)$ 
 is immediate by Definition \ref{def:intrinsic} for  $\chi\in C^1(\RR)$ with $\chi'\in C_b(\RR^d)$,
the statement of Lemma \ref{lem:closable} now follows by the criterion in \cite[Section V.1.3]{BH91} for general image structures.
\end{proof}

The minimal closed extension of the form in Lemma \ref{lem:closable} is denoted by $(\E^{G,\lam},\D(\E^{G,\lam}))$ in the remainder of this text.

\begin{rem}\label{rem:imageForm}
	The proof of Lemma \ref{lem:closable} yields in fact $$\D(\E^{G,\lam})\subseteq \big\{u\in L^2(\P_2, G\circ \Psi_\lam^{-1}): u\circ \Psi_\lam\in W_{A^{-1}}^{1,2}(T_\lam,G)\big\}$$
	and the identity
	\begin{align*}
		\E^{G,\lam}(u,v)={\la u\circ\Psi_\lam,v\circ\Psi_\lam\ra}_{W^{1,2}_{A^{-1}}(T_\lam, G)}\qquad \text{for }u,v\in \D(\E^{G,\lam}),
	\end{align*}
	where $\big({\la \cdot ,\cdot \ra}_{W^{1,2}_{A^{-1}}(T_\lam, G)},W^{1,2}_{A^{-1}}(T_\lam, G)\big)$ denotes the minimal closed extension of \eqref{eq:WAform} in $L^2(T_\lam,G)$.
\end{rem}

\subsection{Support properties   \texorpdfstring{of $G\circ\Psi_\lam^{-1}$}{}}\label{sec:supp}
For $1\leq p<\infty$ we set $\scr P_p:=\{\mu:$ $\mu$ is a probability measure on $\RR^d$ and  $\int_{\RR^d}|\cdot|_\euc^p\d\mu<\infty\}$.
We take $\lam$, $\Psi_\lam$ and $G$ as in Section \ref{sec:B}.

Here, we discuss some
essential properties of $G\circ\Psi_\lam^{-1}$, which  is the invariant measure of a diffusion process ${(\bm\mu_t)}_{t\geq 0}$ on $\P_2$ analyzed in Section \ref{sec:Fuku}, namely
Properties (G1), (G2), (G3) and (G4) as follows:
\begin{enumerate}
	\item[(G1)] The initial choice of $\lam$ is irrelevant.
	In fact, for any  $\mu\in\P^\ac$ we can define
	\begin{equation*}
		\Psi_\mu:T_\mu\to\scr P_2,\qquad f\mapsto \mu\circ f^{-1},
	\end{equation*}
	and find a  Gaussian measure $\tilde G$ on $T_\mu$ such that
	\begin{equation*}
		G\circ\Psi_\lam^{-1}=\tilde G\circ\Psi_{\mu}^{-1}
	\end{equation*}
	as well as
	\begin{equation*}
		\big(\E^{G,\lam},\D(\E^{G,\lam})\big)=\big(\E^{\tilde G,\mu},\D(\E^{\tilde G,\mu})\big).
	\end{equation*}
	\item[(G2)] $G\circ\Psi_\lam^{-1}$ has full topological support on $\P_2$ (see Remark \ref{rem:fts}).
\end{enumerate}
The two items (G1), (G2) are satisfied for any choices of $\lam\in \P^\ac$ and a non-degenerate Gaussian $G$ on $T_\lam$.	
Additionally, we like to have that
\begin{enumerate}
\item[(G3)]
	$\mu\mapsto \int_{\RR^d}|\cdot|_\euc^p\d\mu$ is $q$-integrable w.r.t.~$G\circ\Psi_\lam^{-1}$ for all $p,q\in[1,\infty)$,
\end{enumerate}
which is equivalent to the statement
\begin{equation*}
	f\mapsto \int_{\RR^d}|f|_\euc^p\d\lam\qquad\text{belongs to}\quad L^q(T_\lam,G).
\end{equation*}
The latter basically states that for every $1\leq p<\infty$, 
$G$ is a Gaussian measure on the Banach space $L^p(\RR^d,\lam;\RR^d)$ of $p$-integrable vector fields on $\RR^d$. Sufficient conditions on $G$ can be found in \cite{B98}.
An example for $G$ is given below (Example \ref{exa:rbf}).
Condition (G3) allows us to consider moments of the process ${({\bm\mu}_t)}_{t\geq 0}$ of arbitrary order. However, it is not necessary for the existence of 
${({\bm\mu}_t)}_{t\geq 0}$ and its second moment.

Moreover, 
\begin{enumerate}	
	\item[(G4)] $G\circ\Psi_\lam^{-1}(\P^\ac)=1$ holds under the condition in Lemma \ref{lem:pac} below.
\end{enumerate}

Let $\Lambda:= G\circ\Psi_\lam^{-1}$. The first property is shown by the following observation.
\begin{rem}
\begin{itemize}[leftmargin=*]
	\item[]a) Let $\mu\in\P^\ac$. There is a transport map $h:\RR^d\to\RR^d$ such that $\lam=\mu\circ h^{-1}$. The map
	\begin{equation*}
		I_h:T_\lam\to T_\mu,\quad I_hf:=f\circ h,
	\end{equation*}
	is a linear isometry and yields $\Psi_\lam=\Psi_\mu\circ I_h$. So, defining a Gaussian measure $\tilde G:=G\circ I_h^{-1}$ on $T_\mu$ results in $\Lam=\tilde G\circ\Psi_\mu^{-1}$.
	\item[]b) 
	Let $I_h^*:T_\mu\to T_\lam$ denote the adjoint operator of $I_h$ (identifying a Hilbert space with its dual space).
	The representation of the Dirichlet form $\E^{G,\lam}(u,v)$, for $u,v\in C_b^1(\P_2)$, in terms of $\tilde G$ and the covariance operator $$\tilde A:=I_hA^{-1}I_h^*$$ of $\tilde G$ is obtained from the chain rule. As in the proof of Lemma \ref{lem:closable}, let $\na$ denote the gradient on $T_\lam$. To distinguish in notation between gradients w.r.t.~different Hilbert spaces, we write
	$\na_{T_\mu}$ for the gradient on $T_\mu$  in the subsequent series of identities. 
	
	If $w:T_\mu\to\RR$ is Fréchet differentiable at $I_hf$ for some $f\in T_\lam$, then $w\circ I_h:T_\lam\to\RR$ is Fréchet differentiable at $f$ and 
	\begin{equation*}
		{\big\la\na(w\circ I_h)(f),g\big\ra}_{T_\lam}={\big\la \na_{T_\mu}w(I_hf),I_hg\big\ra}_{T_\mu}\qquad\text{for }g\in T_\lam.
	\end{equation*}
	The analogue of \eqref{eq:chainR}, i.e.~the chain rule for the composition of a $C_b^1(\P_2)$-function with the map $T_\nu\ni f\mapsto\nu\circ f^{-1}$ can be stated w.r.t.~any reference point  $\nu\in \P_2$.
	In particular,
	\begin{equation*}
		\na(u\circ\Psi_\mu)(f)=Du(\Psi_\mu(f))\circ f,\qquad f\in T_\mu,\,u\in C_b^1(\P_2).
	\end{equation*}
	By the above considerations, for $u,v\in C_b^1(\P_2)$ we have
	\begin{align*}
		&{\big\la \na (u\circ\Psi_\lam),A^{-1}\na (v\circ\Psi_\lam)\big\ra}_{T_\lam}
	={\big\la \na (u\circ\Psi_\mu\circ I_h),A^{-1}\na (v\circ\Psi_\lam)\big\ra}_{T_\lam}\\
		&={\big\la \na_{T_\mu} (u\circ\Psi_\mu)\circ I_h,I_hA^{-1}\na (v\circ \Psi_\lam)\big\ra}_{T_\mu}\\
		&={\big\la A^{-1}I_h^* \big(\na_{T_\mu}(u\circ\Psi_\mu)\circ I_h\big),\na (v\circ \Psi_\lam)\big\ra}_{T_\lam}\\
		&={\big\la A^{-1}I_h^* \big(\na_{T_\mu}(u\circ\Psi_\mu)\circ I_h\big),\na (v\circ \Psi_\mu\circ I_h)\big\ra}_{T_\lam}\\
		&={\big\la \tilde A \big(\na_{T_\mu}(u\circ\Psi_\mu)\circ I_h\big),\na_{T_\mu} (v\circ\Psi_\mu)\circ I_h\big\ra}_{T_\mu}.
	\end{align*}
	Consequently,
		\begin{align*}
		\E^{G,\lam}(u,v)&=\int_{T_\lam}{\big\la \na (u\circ\Psi_\lam)(f),A^{-1}\na (v\circ\Psi_\lam)(f)\big\ra}_{T_\lam}G(\d f)\\
		&=\int_{T_\lam}{\big\la \tilde A\na_{T_\mu} (u\circ\Psi_\mu) (I_hf),\na_{T_\mu} (v\circ\Psi_\mu)(I_hf)\big\ra}_{T_\mu}G(\d f)\\
		&=\underbrace{\int_{T_\mu}{\big\la \na_{T_\mu} (u\circ\Psi_\mu) (f),\tilde A\na_{T_\mu} (v\circ\Psi_\mu)(f)\big\ra}_{T_\mu}\tilde G(\d f)}_{=:\E^{\tilde G,\mu}(u,v)}.
	\end{align*}
\end{itemize}
\end{rem}

\begin{rem}\label{rem:fts}
By definition $G$ is a non-degenerate Gaussian measure on $T_\lam$, i.e.~having full topological support on $T_\lam$
(this is a consequence of \cite[Thm.~1.2.1]{DZ02} and the fact that all eigenvalues of the covariance operator $A^{-1}$ are strictly positive).
Since the pre-image set under $\Psi_\lam$  of any non-empty, open subset $U$ of $\P_2$ is again non-empty and open, $\Lam(U)=G(\Psi_\lam^{-1}(U))>0$.
So, $\Lam$ has full topological support on $\P_2$.
\end{rem}

The following example demonstrates how $G$ and $\lambda$ can be chosen explicitly to obtain a suitable reference measure $\Lam=G\circ\Psi_\lam^{-1}$.
\begin{exa}\label{exa:rbf}
	Let $\nu(\d s):=\frac 1{\sqrt \pi} e^{-s^2}\d s$ denote the normal distribution on $\RR$ with mean zero and variance $\frac 1 2$.
	We define a kernel operator 
	$S^{(1)}:L^2(\RR,\nu)\to L^2(\RR,\nu)$,
	\begin{equation*}
		S^{(1)}f(s):=\int_{\RR} f(t)e^{-(s-t)^2} \nu( \d t),\qquad f\in L^2(\RR,\nu).
	\end{equation*}
	Now we set \begin{equation*}\lambda:=\nu^{\otimes d}:=\frac 1{\pi^{d/2}} e^{-|x|_\euc^2}\d x,\end{equation*} where $\d x$ is the $d$-dimensional Lebesgue measure, and define the kernel operator 
	$S:= L^2(\RR^d,\lam)\to L^2(\RR^d,\lam)$,
	\begin{equation*}
		S\varphi(x):=\int_{\RR^d} \varphi(y)e^{-|x-y|_\euc^2} \lam( \d y),\qquad \varphi\in L^2(\RR^d,\lam).
	\end{equation*}
	Then, $S$ is the $d$-fold tensor product of $S^{(1)}$, which acts on factorizable  $\varphi(x)=f^{(1)}(x_1)\cdot\dots\cdot f^{(d)}(x_d)$, $f^{(i)}\in L^2(\RR,\nu)$, as
	\begin{equation*}
		S\varphi(x)=\prod_{i=1}^dS^{(1)}f^{(i)}(x_i).
	\end{equation*}
	It is well-known that the Gaussian kernel $s(x,y):=e^{-|x-y|_\euc^2}$, $x,y\in\RR^d$, is positive definite (see e.g.~\cite{gheon} for a proof).
	Since moreover $s$ is symmetric and $\int_{\RR^d}s(x,x)\lam(\d x)=1$,
	$S$ is a symmetric, positive definite trace class operator on $L^2(\RR^d,\lam)$. 
	
	In fact, the eigenvalues ${(\beta_n)}_{n\in\NN}\subset (0,\infty)$ and  
	a corresponding eigenbasis ${(h_n)}_{n\in\NN}\subset L^2(\RR,\nu)$ of $S^{(1)}$ can be calculated explicitly 
	(see \cite{Zhu, RasWil, fass}). We may assume orthonormality $\langle h_n,h_m\rangle_{L^2(\RR,\nu)}=\delta_{nm}$, $n,m\in\NN$. The countable sets of eigenvalues $\alpha_n$ and their corresponding normalized eigenvectors $\varphi_n$ of $S$ are therefore given by
	\begin{equation*}
		\{\alpha_n:n\in\NN\}:=\Big\{\alpha:(n_1,\dots,n_d)\in\NN^d,\,\alpha=\prod_{i=1}^d\beta_{n_i}\Big\}\subset (0,\infty)
	\end{equation*}
	respectively
	\begin{equation*}
		\{\varphi_n:n\in\NN\}:=\Big\{\varphi: (n_1,\dots,n_d)\in\NN^d,\,\varphi(x)=\prod_{i=1}^dh_{n_i}(x_i)\Big\}\subset L^2(\RR^d,\lam).
	\end{equation*}
	The Mercer expansion for the Gaussian kernel reads
	\begin{equation*}
		s(x,y)=\sum_{n\in\NN}\alpha_n\varphi_n(x)\varphi_n(y).
	\end{equation*}
	Since the kernel $s$ is positive definite, there exists a mean-zero Gaussian random field $\{\xi_x\}_{x\in\RR^d}$ whose covariance function equals $s$, i.e.
	$\EE[\xi_x\xi_y]=s(x,y)$, $x,y\in\RR^d$ (see e.g.~\cite[Chap.~IV]{GS74}).
	What is more, $\int_{\RR^d}s(x,x)^{p /2} \lam(\d x)=1$ implies that the law of ${(\xi_x)}_{x\in\RR}$ defines a Gaussian measure on
	$L^p(\RR^d,\lam)$ for every $1\leq p<\infty$ by virtue of in \cite[Exa.~2.3.16]{B98}, and moreover
	\begin{equation}\label{eq:XiPQ}
		\EE\Big[\Big(\int_{\RR^d} |\xi_x|^p\lam(\d x)\Big)^q\Big]<\infty,\qquad 1\leq q<\infty.
	\end{equation}
	Due to smoothness of the Gaussian kernel $s$, the sample paths are infinitely differentiable, i.e.~$x\mapsto \xi_x$ belongs to $C^\infty(\RR^d)$ almost surely 
	(for a proof, see \cite{scheu} or \cite{daC}).
	
	Now, we take $d$ independent copies $\{\xi^{(i)}_x\}_{x\in\RR^d}$, $i=1,\dots,d$, of $\{\xi_x\}_{x\in\RR^d}$ and define a $\RR^d$-valued Gaussian random field 
	\begin{equation*}
		\vec\xi_x:=\big(\xi_x^{(1)},\dots, \xi_x^{(d)}\big)^\tr\in\RR^d,\qquad x\in\RR^d,
	\end{equation*}
	with sample paths in $C^\infty(\RR^d,\RR^d)\cap T_\lam$.
	The law of ${(\vec\xi_x)}_{x\in\RR^d}$ yields a Gaussian $G$ on $T_\lam$, whose covariance operator on $T_\lam$ is the inverse $A^{-1}$ of a strictly positive, self-adjoint  operator $(A,\D(A))$, and
	denoting the components of an element $f\in T_\lam$ by $f_i\in L^2(\RR^d,\lam)$, $i=1,\dots,d$, it holds
	\begin{equation}\label{eq:AS}
		A^{-1}f=(Sf_1,\dots,Sf_d)^\tr\in{\big( L^2(\RR^d,\lam)\big)}^d\simeq T_\lam,\qquad f\in T_\lam\simeq{\big(L^2(\RR^d,\lam)\big)}^d.
	\end{equation}
	Using \eqref{eq:XiPQ} we conclude for $2\leq p<\infty$ and $1\leq q<\infty$,
	\begin{multline*}
		\int_{T_\lam}\big\||f|_\euc^p\big\|_{L^1(\RR^d,\lam)}^qG(\d f)\leq d^{\frac {p q}{2}-1} \sum_{i=1}^d\int_{T_\lam}\Big(\int_{\RR^d}|f_i(x)|^p\lam(\d x)\Big)^qG(\d f)\\
		= d^{\frac {p q}{2}}\EE\Big[\Big(\int_{\RR^d} |\xi_x|^p\lam(\d x)\Big)^q\Big]<\infty.
	\end{multline*}
	In particular, Property (G3) is satisfied and $G(C^\infty(\RR^d,\RR^d))=1$.
\end{exa}

Next, we discuss property (G4). By the considerations above, we may restrict to any choice  of $\lam\in\P^\ac$. 
In particular, let us assume  the full topological support of $\lam$, i.e.
\begin{equation}\label{eq:supplam}
\supp[\lam]=\RR^d,
\end{equation}
for simplicity.
An element $f\in T_\lam$ can have at most one $\lam$-version  belonging to the class of continuously differentiable maps $C^1(\RR^d,\RR^d)$, which we again denote by $f$ if existent.
For a reasonable class of Gaussian measures $G$ on $T_\lam$ it holds $\Lam(\P^\ac)=G (\Psi_\lam^{-1}(\P^\ac))=1$.
As we show in the following, a sufficient condition on $G$ reads:
 $G$-a.e.~$f$ satisfies $ f\in C^1(\RR^d,\RR^d)$ and its Jacobi matrix $Jf(x)$ is invertible at $\lam$-a.e.~point $x\in\RR^d$, i.e.
\begin{equation}\label{eq:regG}
	G\Big(\{f\in C^1(\RR^d,\RR^d):\lambda(\{\det( Jf)=0\})=0\big\}\Big)=1.
\end{equation}

\begin{rem}\label{rem:pac}
	\begin{itemize}[leftmargin=*]
		\item[]a) Let $f\in C^1(\RR^d,\RR^d)$ and $\rho_\lam$ denote the probability density of $\lam$ w.r.t.~the Lebesgue measure. 
		For an open domain $O\subset\RR^d$ on which  $f$ restricts to a diffeomorphism  $f|_O:O\to f(O)$, we have
		\begin{equation}\label{eq:trafo}
			\int_O \varphi(f(x))\lam(\d x)=\int_{f(O)}\frac{\varphi(x)\rho_\lam(f^{-1}(x))}{\big|\det (Jf(f^{-1}(x)))\big|}\d x
		\end{equation} 
		for measurable $\varphi:=f(O)\to[0,\infty)$.
		\item[]b) Let $f\in C^1(\RR^d,\RR^d)$ with $\lambda(\{\det (Jf)=0\})=0$. 
		The open set $\{\det( Jf)\neq 0\}$ can be written as the countable union of open domains ${\{O_n\}}_{n\in\NN}$ on each of which $f$ restricts to a diffeomorphism
		(Inverse Function Theorem). If $N\subset\RR^d$ is measurable and has Lebesgue measure zero, then
		\begin{equation*}
			\lam(f^{-1}(N))=\lam\big(f^{-1}(N)\cap \{\det( Jf)\neq 0\}\big)\leq \sum_{n\in\NN}\int_{O_n}\eins_{N}(f(x)) \lam(\d x)=0
		\end{equation*}
		by \eqref{eq:trafo}. Hence $\Psi_\lam(f)=\lam\circ f^{-1}\in\P^\ac$.
	\item[]c) Under Assumptions \eqref{eq:supplam} \& \eqref{eq:regG}, $\Lam(\P^\ac)=G (\Psi_\lam^{-1}(\P^\ac))=1$ follows by b).
\end{itemize}
\end{rem}

\begin{lem}\label{lem:pac}
	Let $G$ be a non-degenerate Gaussian measure on $T_\lam:=L^2(\RR^d,\lam;\RR^d)$ as in \eqref{eq:covG}. We assume
	\eqref{eq:supplam} and
	\begin{equation*}G\big(C^1(\RR^d,\RR^d)\big)=1.\end{equation*} 
	Then,  \eqref{eq:regG} is satisfied. In particular (by Remark \ref{rem:pac}),
	\begin{equation*}
		G (\Psi_\lam^{-1}(\P^\ac))=1.
	\end{equation*}
\end{lem}
\begin{proof}
	At first we comment that $G$ can be regarded as a non-degenerate Gaussian measure on the Fréchet-space $C^1(\RR^d,\RR^d)$ (which is equipped with the topology of local uniform convergence of functions together with local uniform convergence of their first-order partial derivatives). This comment follows from the observations (recalling that $\lam$ has a density w.r.t.~the Lebesgue measure which is strictly positive a.e.):
	\begin{itemize}
	\item[a)] $T_\lam\cap C^1(\RR^d,\RR^d)$ is a Borel measurable and dense subset of $C^1(\RR^d,\RR^d)$, 
	\item[b)] the family of maps 
	\begin{equation}\label{eq:weak*dense}
		C^1(\RR^d,\RR^d)\ni f\mapsto \int \la f,\kappa\ra_\euc d\lam\in\RR,\quad \kappa:\RR^d\to\RR^d\text{ bounded, measurable, compactly supported},
	\end{equation}
	generate the Borel $\sigma$-algebra of $C^1(\RR^d,\RR^d)$.
	\end{itemize}
	Hence, by extending trivially with $G(C^1(\RR^d,\RR^d)\setminus T_\lam):=0$, $G$ is well-defined on the Borel $\sigma$-algebra of $C^1(\RR^d,\RR^d)$.
	Moreover, $G$ is also Gaussian on $C^1(\RR^d,\RR^d)$, since the family in \eqref{eq:weak*dense} is weak*-dense in the topological dual space $C^1(\RR^d,\RR^d)'$.
	In this context we denote by  \begin{equation*}K:{(C^1(\RR^d,\RR^d))}'\to C^1(\RR^d,\RR^d)\end{equation*} the covariance operator of $G$, i.e.~
	\begin{equation*}
		\int_{C^1(\RR^d,\RR^d)} g(f)h(f)G(\d f)=g(Kh),\qquad g,h\in {(C^1(\RR^d,\RR^d))}'.
	\end{equation*}

	Next, let $x\in\RR^d$ and define the continuous linear map 
	\begin{equation*}J_x:C^1(\RR^d,\RR^d)\ni f\mapsto Jf(x)\in\RR^{d\times d}\simeq\RR^{d^2}
	\end{equation*}
	with adjoint operator $J_x^*:{(\RR^{d\times d})}'\simeq\RR^{d^2}\to {C^1(\RR^d,\RR^d)}'$ given through the equation
	\begin{equation*}
		J_x^*y (f):=\la y,J_xf\ra_\euc,\qquad f\in C^1(\RR^d,\RR^d),\,y\in \RR^{d^2}.
	\end{equation*} 
	The image measure $G\circ J_x^{-1}$ is the non-degenerate centered Gaussian measure on $\RR^{d^2}$ 
	with covariance $J_xKJ_x^*$
	and therefore absolutely continuous w.r.t.~the $d^2$-dimensional Lebesgue measure. 
	We denote by $\Theta$ the non-zero multivariate polynomial function $\Theta:\RR^{d^2}\to\RR$ 
	which corresponds to the determinant under the above identification of $\RR^{d\times d}$ and $\RR^{d^2}$.
	The zero level set $\{y\in\RR^{d^2}:\Theta(y)=0\}$ is a lower dimensional manifold embedded into $\RR^{d^2}$ and hence a negligible set w.r.t.~$G\circ J_x^{-1}$.

	By Fubini's theorem and the argumentation above,
	\begin{multline*}
		\int_{T_\lam} \int_{\RR^d}\eins_{\{0\}}\big(\det (Jf(x))\big)\lam(\d x)G(\d f)
		= \int_{\RR^d}\int_{T_\lam}\eins_{\{0\}}\big(\det (Jf(x))\big)G(\d f)\lam(\d x)\\
		= \int_{\RR^d}\int_{\RR^{d^2}}\eins_{\{0\}}\big(\Theta (y)\big)G\circ J_x^{-1}(\d y)\lam(\d x)=0.
	\end{multline*}
	This implies \eqref{eq:regG}.
\end{proof}

\section{The Ornstein--Uhlenbeck operator on \texorpdfstring{$\scr P_2$}{P2} }\label{sec:generator}
In this section we calculate the generator $L_{G,\lam}$ of $(\E^{G,\lam},\D(\E^{G,\lam}))$ with a Gaussian measure $G$ and $\lam\in\P^\ac$ as in Section \ref{sec:B}.
It is the generator of a Markov process on $\P_2$, which is analyzed in Section \ref{sec:Fuku}.
The main results of this section are formulated in Subsections \ref{sec:gen} \& \ref{sec:diffOp}:
\begin{itemize}
	\item Theorem \ref{thm:thmL} characterizes how $L_{G,\lam}$ acts on functions $u:\scr P_2\to\RR$ of cylindrical shape as given in \eqref{eq:u} \& \eqref{eq:ubounds}
	\item If $u$ is taken as the evaluation map $u:\mu\mapsto \mu(\varphi)$ for a test function $\varphi\in C^2(\RR^d)$, then
	$L_{G,\lam}$ acts on $u$ as a second-order differential operator in $\varphi$ with $\mu$-dependent coefficients (see \eqref{eq:Ldiff} and Proposition \ref{prp:diffO}).
\end{itemize}

\subsection{Partial integration on the tangent space \texorpdfstring{$T_\lam$}{}}\label{sec:AB}
We recall the map $\Psi_\lam$ from \eqref{eq:DefPsi}.
As a preparation for Theorem \ref{thm:thmL} and Proposition \ref{prp:diffO} we compute the divergence of the vector field $B\na (u\circ\Psi_\lam)$ on $T_\lam$ w.r.t.~the Gaussian measure $G$,
where $u$ has a cylindrical shape as in \eqref{eq:u} \& \eqref{eq:ubounds} and $B$ is a symmetric trace class operator on $T_\lam$. The result is formulated in
Proposition \ref{prp:divAB}.

For $w\in W^{1,2}(T_\lam,G)$, the $(1,2)$-Sobolev space (see \cite[Chap.~II.9.2]{DZ02}) w.r.t.~the Gaussian measure $G$, and $g\in T_\lam\setminus\{0\}$ we set
\begin{equation*}
	\partial_g w:=\la g,\na w(\cdot)\ra_{T_\lam}\in L^2(T_\lam,G).
\end{equation*}
We now fix $w\in W^{1,2}(T_\lam,G)$ such that additionally $ \partial_gw\in W^{1,2}(T_\lam,G)$ for all $g\in T_\lam\setminus\{0\}$ and
\begin{equation}\label{eq:rDef}
	\|\na w\|_{T_\lam}\in L^r(T_\lam,G)\quad\text{for some }r\in(2,\infty).
\end{equation}

Moreover, we consider a symmetric trace class operator $B$ on $T_\lam$ such that
\begin{equation}\label{eq:divergence}
\div (B\na w):=	\sum_{n\in\NN} \partial_{g_n} \la g_n,B\na w(\cdot)\ra_{T_\lam}\in   L^2(T_\lam,G)
\end{equation}
for any orthonormal basis ${(g_n)}_{n\in\NN}$ of $T_\lam$. 
The function on the right-hand side of \eqref{eq:divergence} doesn't depend on the choice of an orthonormal basis and 
is the divergence of the vector field $B\na w$.

Let $B$ be chosen in such a way that the composition $AB$ is a bounded operator on $T_\lam$, where
$A^{-1}$ is the covariance operator of $G$, as in Section \ref{sec:Dform}.
 Integration by parts for the Gaussian measure $G$ (see the subsequent remark for details) yields
\begin{equation}\label{eq:AB}
	\int_{T_\lam}{\big\la\na v(f),B\na w(f)\big\ra}_{T_\lam} G(\d f)
	=-\int_{T_\lam}v(f)\Big[\div (B\na w)(f)-{\big\la AB\na w(f),f\big\ra}_{T_\lam}\Big]G(\d f)
\end{equation}
for $w$ as above and $v\in W^{1,2}(T_\lam,G)$.
\begin{rem}
	In fact, \eqref{eq:AB} is obtained from \cite[Cor.~9.2.9]{DZ02} by choosing an eigenbasis ${(g_n)}_{n\in\NN}$ of $A$, since
	\begin{align}\label{eq:GparI}
		&\int_{T_\lam}{\big\la\na v(f),B\na w(f)\big\ra}_{T_\lam} G(\d f)\\\nonumber
		&=\sum_{n\in\NN}\int_{T_\lam}\partial_{g_n} v(f){\big\la g_n,B\na w(f)\big\ra}_{T_\lam} G(\d f)\\
		&=-\sum_{n\in\NN}\int_{T_\lam}v(f)\Big[\partial_{g_n}{\big\la g_n,B\na w\big\ra}_{T_\lam}(f)-\nonumber
	{\big\la Ag_n,f\big\ra}_{T_\lam}	{\big\la g_n, B\na w(f)\big\ra}_{T_\lam}\Big]G(\d f).
	\end{align}
	We have ${\la AB\na w(\cdot ),\,\cdot\, \ra}_{T_\lam}\in L^2(T_\lam,G)$,
	because $\|AB\na w\|_{T_\lam}\leq C\|\na w\|_{T_\lam}\in L^r(T_\lam,G)$ for some $r\in(2,\infty)$, $C\in(0,\infty)$ by our assumptions, and the Gaussian measure possesses all moments
	\begin{equation}\label{eq:Gmom}
		\int_{T_\lam}\|f\|_{T_\lam}^qG(\d f)<\infty,\qquad q\in[1,\infty).
	\end{equation}
	Indeed, by the Cauchy-Schwarz inequality, with $r$ as in \eqref{eq:rDef} and $q:=\frac {2r}{r-2}$, 
	\begin{align*}
	\int_{T_\lam}{\la AB\na w(f ),f \ra}_{T_\lam}^2G(\d f)&\leq C\int_{T_\lam}\|\na w(f)\|_{T_\lam}^2\|f\|_{T_\lam}^2G(\d f) 
	\\&\leq C\big\|\|\na w\|_{T_\lam}\big\|_{L^r(T_\lam,G)}^{2 }\big\|\|f\|_{T_\lam}\big\|_{L^q(T_\lam,G)}^{2}.
	\end{align*}
	So, \eqref{eq:GparI} equals
	\begin{equation*}
	-\int_{T_\lam}v(f)\Big[\div (B\na w)(f)-{\big\la AB\na w(f),f\big\ra}_{T_\lam}\Big]G(\d f)
	\end{equation*}
	as claimed in \eqref{eq:AB}.
\end{rem}

In the subsequent proposition we specify the term $\div (B\na w)$ as appearing in \eqref{eq:AB} for a certain type of functions $w$ which are relevant in our context.
For a probability measure $\mu$ on $\RR^d$ and a $\mu$-integrable function $\varphi:\RR^d\to\RR$ we write $\mu(\varphi):=\int_{\RR^d}\varphi\d\mu$.
Proposition \ref{prp:divAB} states the explicit calculation of $\div (B\na w)$ in case $w=u\circ\Psi_\lam$ for $u\in C^1(\P_2)$ with cylindrical shape. More precisely, we assume
\begin{equation}\label{eq:u}
	u:\scr P_2\ni \mu\mapsto F\big(\mu(\varphi_1),\dots,\mu(\varphi_m)\big)\in\RR
\end{equation}
with coefficients $F\in C^2(\RR^m)$, $\varphi_1,\dots,\varphi_m\in C^2(\RR^d)$, $m\in\NN$, such that 
\begin{align}\label{eq:ubounds}
\sup_{y\in\RR^m}	\big(|\partial_i F(y)|+|\partial_i\partial_j F(y)|\big)+\sup_{x\in\RR^d}|\partial_k\partial_l \varphi_i(x)|<\infty,\qquad 1\leq i,j\leq m,\quad 1\leq k,l\leq d.
\end{align}
The gradient of $\varphi_i$ is denoted by $\na^\euc\varphi_i:\RR^d\to\RR^d$ in order to distinguish in notation from the gradient for functions in $W^{1,2}(T_\lam,G)$. 
Due to \eqref{eq:ubounds}, 
\begin{equation}\label{eq:phibounds}
	\sup_{x\in\RR^d}\frac{|\na^\euc \varphi_i(x)|_\euc}{1+|x|_\euc}<\infty\qquad \text{and}\qquad\sup_{x\in\RR^d}\frac{|\varphi_i(x)|}{1+|x|_\euc^2}<\infty
\end{equation}
for $1\leq i\leq m$. In particular, \eqref{eq:u} is well-defined and continuous.
The Hessian matrix of $\varphi_i$ is denoted by 
\begin{equation*}
	\hes\varphi_i:={\big[\partial_k\partial_l \varphi_i\big]}_{k,l=1}^d.
\end{equation*}
For $f,g\in T_\lam$ the  matrix-vector product
\begin{equation*}
	\RR^d\ni x\,\mapsto\, \hes\varphi_1  (f(x))g(x)\in\RR^d
\end{equation*}
yields again an element in $T_\lam$, which we denote by $(\hes\varphi_1 \circ f)g$, and $(\hes\varphi_1 \circ f):T_\lam\to T_\lam$ is a bounded linear operator.


\begin{prp}\label{prp:divAB}
	Let $u$ be given as in \eqref{eq:u}, \eqref{eq:ubounds}. The following are true:
	
	\begin{equation*}u\circ\Psi_\lam,\, \partial_g (u\circ\Psi_\lam) \in W^{1,2}(T_\lam,G)\quad\text{for }g\in T_\lam\setminus\{0\}\end{equation*}
	and
	\begin{multline*}
		\na\partial_g(u\circ\Psi_\lam)(f)=\Big(\sum_{i,j=1}^m\partial_i\partial_jF\big(\lam(\varphi_1\circ f),\dots,\lam(\varphi_m\circ f)\big)
	 {\big\la \na^\euc\varphi_j \circ f,g\big\ra}_{T_\lam}\na^\euc\varphi_i \circ f\Big)
	\\+\Big(\sum_{i=1}^m\partial_iF\big(\lam(\varphi_1\circ f),\dots,\lam(\varphi_m\circ f)\big)(\hes\varphi_i \circ f)g\Big).
	\end{multline*}
	for $f\in T_\lam$. Moreover,
	\begin{multline*}
		\div \big(B\na(u\circ\Psi_\lam)\big)(f)=
		\Big(\sum_{i,j=1}^m\partial_i\partial_jF\big(\lam(\varphi_1\circ f),\dots,\lam(\varphi_m\circ f)\big)
		{\big\la B(\na^\euc\varphi_i \circ f),\na^\euc\varphi_j \circ f\big\ra}_{T_\lam}\Big)\\
		+\Big(\sum_{n\in\NN}\sum_{i=1}^m\partial_iF\big(\lam(\varphi_1\circ f),\dots,\lam(\varphi_m\circ f)\big){\big\la Bg_n,(\hes\varphi_i \circ f)g_n\big\ra}_{T_\lam}\Big)
	\end{multline*}
	for any symmetric trace class operator $B$ on $T_\lam$, orthonormal basis ${(g_n)}_{n\in\NN}$ and $f\in T_\lam$. In particular, $\div (B\na(u\circ\Psi_\lam))\in L^2(T_\lam,G)$.
	
\end{prp}

\begin{proof}
The proof is divided into three steps, in which we successively verify 
$u\circ\Psi_\lam \in W^{1,2}(T_\lam,G)$ (step 1), $\partial_g (u\circ\Psi_\lam) \in W^{1,2}(T_\lam,G)$ together with claimed formula for $\na\partial_g(u\circ\Psi_\lam)(f)$ (step 2),
and finally the expression for $\div (B\na(u\circ\Psi_\lam))$ as claimed (step 3).

Step 1: Using \eqref{eq:phibounds} and Lebesgue's Dominated Convergence, it holds for $f,g\in T_\lam$,
\begin{align}\label{eq:nauPsi}
	\frac{\d}{\d\varepsilon}(u\circ\Psi_\lam)(f+\varepsilon g)\Big|_{\varepsilon=0}&=\frac{\d}{\d\varepsilon}F\big(\lam(\varphi_1\circ (f+\varepsilon g)),\dots,\lam(\varphi_m\circ (f+\varepsilon g))\big)\Big|_{\varepsilon=0}\\
	&=\sum_{i=1}^m\partial_i F\big(\lam(\varphi_1\circ f),\dots,\lam(\varphi_m\circ f)\big)\int_{\RR^d}\frac{\d}{\d\varepsilon} \big(\varphi_i\circ (f+\varepsilon g)\big)\Big|_{\varepsilon=0}\d\lam\nonumber\\
	&=\sum_{i=1}^m\partial_i F\big(\lam(\varphi_1\circ f),\dots,\lam(\varphi_m\circ f)\big)\lam\big(\la \na^\euc\varphi_i\circ f,g\ra_\euc \big).\nonumber
\end{align}
Due to \eqref{eq:nauPsi}, \eqref{eq:phibounds} and \eqref{eq:Gmom}
we have $\|\na (u\circ\Psi_\lam)\|_{T_\lam}\in L^2(T_\lam,G)$, where
\begin{equation}\label{eq:nauPsi2}
	\na (u\circ\Psi_\lam):T_\lam\ni f\mapsto \sum_{i=1}^m\partial_i F\big(\lam(\varphi_1\circ f),\dots,\lam(\varphi_m\circ f)\big) (\na^\euc\varphi_i\circ f) \in T_\lam.
\end{equation}
In particular, $u\circ\Psi_\lam \in W^{1,2}(T_\lam,G)$ and $\|\na (u\circ\Psi_\lam)\|_{T_\lam} \in L^r(T_\lam,G)$ for any $r\in[1,\infty)$.

Step 2: Using \eqref{eq:ubounds}, \eqref{eq:phibounds}, Lebesgue's Dominated Convergence and the product rule, it holds for $f,g,h\in T_\lam$,
\begin{align*}
	&\frac{\d}{\d\varepsilon}\partial_g(u\circ\Psi_\lam)(f+\varepsilon h)\Big|_{\varepsilon=0}\\
	&=\frac{\d}{\d\varepsilon}\sum_{i=1}^m \partial_iF\big(\lam(\varphi_1\circ (f+\varepsilon h)),\dots,\lam(\varphi_m\circ (f+\varepsilon h))\big)\la \na^\euc\varphi_i\circ (f+\varepsilon h),g\ra_{T_\lam} \Big|_{\varepsilon=0}\\
	&=\Big(\sum_{j,i=1}^m \partial_j\partial_iF\big(\lam(\varphi_1\circ f),\dots,\lam\big(\varphi_m\circ f)\big)
	\la \na^\euc\varphi_i\circ f,h\ra_{T_\lam}\la \na^\euc\varphi_j\circ f,g\ra_{T_\lam}\Big)\\
	&\quad+\Big(\sum_{i=1}^m\partial_iF\big(\lam(\varphi_1\circ f),\dots,\lam\big(\varphi_m\circ f)\big)\int_{\RR^d}\frac{\d}{\d\varepsilon} \big\la \na^\euc \varphi_i\circ (f+\varepsilon h)),g\big\ra_\euc\Big|_{\varepsilon=0}\d\lam\Big)\\
	&=\Big(\sum_{j,i=1}^m \partial_j\partial_iF\big(\lam(\varphi_1\circ f),\dots,\lam\big(\varphi_m\circ f)\big)
	\la \na^\euc\varphi_i\circ f,h\ra_{T_\lam}\la \na^\euc\varphi_j\circ f,g\ra_{T_\lam}\Big)\\
	&\quad +\Big(\sum_{i=1}^m\partial_iF\big(\lam(\varphi_1\circ f),\dots,\lam\big(\varphi_m\circ f)\big)\int_{\RR^d} \big\la (\hes \varphi_i\circ  f)h,g\big\ra_\euc\d\lam\Big)
\end{align*}
This equation together with \eqref{eq:ubounds}, \eqref{eq:phibounds} and \eqref{eq:Gmom} proves  $\|\na\partial_g(u\circ\Psi_\lam)\|_{T_\lam}\in L^2(T_\lam,G)$, where
\begin{multline*}
	\na\partial_g (u\circ\Psi_\lam)(f)=\Big(\sum_{j,i=1}^m \partial_j\partial_iF\big(\lam(\varphi_1\circ f),\dots,\lam\big(\varphi_m\circ f)\big)
	\la \na^\euc\varphi_j\circ f,g\ra_{T_\lam} \na^\euc\varphi_i\circ f\Big)\\
	\quad +\Big(\sum_{i=1}^m\partial_iF\big(\lam(\varphi_1\circ f),\dots,\lam\big(\varphi_m\circ f)\big)  (\hes \varphi_i\circ  f)g\Big).
\end{multline*}
Hence $\partial_g(u\circ\Psi_\lam)\in W^{1,2}(T_\lam,G)$.

Step 3: Calculating \eqref{eq:divergence} for $w:=u\circ\Psi_\lam$ yields
\begin{align*}
	&\div \big(B\na(u\circ\Psi_\lam)\big)(f)=\sum_{n\in\NN}\big\la Bg_n,\na\partial_{g_n}(u\circ\Psi_\lam)\big\ra_{T_\lam}\\
	&=\Big(\sum_{n\in\NN}\sum_{i,j=1}^m\partial_i\partial_jF\big(\lam(\varphi_1\circ f),\dots,\lam(\varphi_m\circ f)\big)
	{\big\la \na^\euc\varphi_i \circ f,Bg_n\big\ra}_{T_\lam}{\big\la\na^\euc\varphi_j \circ f,g_n\big\ra}_{T_\lam}\Big)\\
	&\qquad+\Big(\sum_{n\in\NN}\sum_{i=1}^m\partial_iF\big(\lam(\varphi_1\circ f),\dots,\lam(\varphi_m\circ f)\big){\big\la B g_n,(\hes\varphi_i \circ f)g_n\big\ra}_{T_\lam}\Big)\\
	&=\Big(\sum_{i,j=1}^m\partial_i\partial_jF\big(\lam(\varphi_1\circ f),\dots,\lam(\varphi_m\circ f)\big)
	{\big\la B(\na^\euc\varphi_i \circ f),\na^\euc\varphi_j \circ f\big\ra}_{T_\lam}\Big)\\
	&\qquad+\Big(\sum_{n\in\NN}\sum_{i=1}^m\partial_iF\big(\lam(\varphi_1\circ f),\dots,\lam(\varphi_m\circ f)\big){\big\la Bg_n,(\hes\varphi_i \circ f)g_n\big\ra}_{T_\lam}\Big)
\end{align*}
as desired. Clearly, $\div (B\na(u\circ\Psi_\lam))\in L^2(T_\lam,G)$ because of \eqref{eq:ubounds}, \eqref{eq:phibounds} and the fact that
 trace class operators form an ideal in the algebra of bounded operators on $T_\lam$ (see e.g.~\cite[Chap.~VI]{sire}).
 Indeed, due to \eqref{eq:ubounds}, the trace of the operator $T_\lam\ni g\mapsto B(\hes\varphi_i \circ f)g\in T_\lam$ is bounded uniformly over $f\in T_\lam$ for $i=1,\dots,m$.
\end{proof}

\subsection{The generator \texorpdfstring{$L_{G,\lam}$}{L}  \texorpdfstring{of $\E^{G,\lam}$}{}}\label{sec:gen}
We now calculate the generator of $\E^{G,\lam}$ by applying the formulas of Section \ref{sec:AB}  for the choice $B=A^{-1}$, the covariance operator of $G$.
Let $w\in L^2(T_\lam,G)$. 
The conditional expectation $\EE_G[w|\Psi_\lam]$ onto the $\sigma$-algebra generated by $\Psi_\lam$ coincides with the orthogonal projection of $w$ onto the subspace
\begin{equation*}
	\big\{u\circ\Psi_\lam:u\in L^2(\scr P_2,G\circ\Psi_\lam^{-1})\big\}\subset  L^2(T_\lam,G).
\end{equation*}
We write
\begin{equation*}
	\mu\mapsto \EE_G[w|\Psi_\lam=\mu ]
\end{equation*}
for the measurable function $\scr P_2\to\RR$ (unique in $(G\circ\Psi_\lam^{-1})$-a.e.~sense) such that
\begin{equation*}
	\EE_G[w|\Psi_\lam=\,\cdot\,](\Psi_\lam(f))=\EE_G[w|\Psi_\lam](f),\qquad G\text{-a.e.~}f\in T_\lam.
\end{equation*}

Next, for given $\varphi_1,\varphi_2\in C^2(\RR^d)$ such that $\sup_{x\in\RR^d}|\partial_k\partial_l \varphi_i(x)|<\infty,$ $ 1\leq k,l\leq d$, $1\leq i\leq 2$, we define  $b(\varphi_1,\varphi_2),a(\varphi_1)\in L^2(T_\lam,G)$ as 
\begin{equation*}
	b(\varphi_1,\varphi_2):T_\lam\ni f\mapsto{\big\la (\na^\euc\varphi_1 \circ f),A^{-1}(\na^\euc\varphi_2 \circ f)\big\ra}_{T_\lam}
\end{equation*}
respectively
\begin{equation}\label{eq:hess}
	a(\varphi_1):T_\lam\ni f\mapsto \sum_{n\in\NN}{\big\la A^{-1}g_n,(\hes\varphi_1 \circ f)g_n\big\ra}_{T_\lam},
\end{equation}
where ${(g_n)}_{n\in\NN}$ is an orthonormal basis in $T_\lam$ and $(\hes\varphi_1 \circ f)g_n\in T_\lam$ is the  matrix-vector product
\begin{equation*}
(\hes\varphi_1 \circ f)g_n:\RR^d\ni x\,\mapsto\, \hes\varphi_1  (f(x))g_n(x)\in\RR^d.
\end{equation*}
\eqref{eq:hess} does not depend on the choice of an orthonormal basis as it coincides with the trace of the operator
\begin{equation*}
	T_\lam\ni g\mapsto A^{-1}(\hes\varphi_1 \circ f)g\in T_\lam
\end{equation*}
for $f\in\ T_\lam$.


The main result of Section \ref{sec:generator}, stated in the theorem below, characterizes the generator  $(L_{G,\lam},\D(L_{G,\lam}))$ of $(\E^{G,\lam},\D(\E^{G,\lam}))$ in $L^2(\scr P_2,G\circ\Psi_\lam^{-1})$.

\begin{thm}\label{thm:thmL}
	If $u$ is given as in \eqref{eq:u}, \eqref{eq:ubounds}, then $u\in \D(L_{G,\lam})$ and 
	\begin{align}\label{eq:generator}
		L_{G,\lam}u(\mu)=&\Big(\sum_{i=1}^m\partial_iF\big(\mu(\varphi_1),\dots,\mu(\varphi_m)\big)\big(\EE_G[a(\varphi_i)|\Psi_\lam=\mu]-{\la\id_{\RR^d},\na\varphi_i\ra}_{T_\mu}\big)\Big)\\
		&+\Big(\sum_{i,j=1}^m\partial_i\partial_jF\big(\mu(\varphi_1),\dots,\mu(\varphi_m)\big)\EE_G[b(\varphi_i,\varphi_j)|\Psi_\lam=\mu]\Big),\nonumber
	\end{align}
	$(G\circ\Psi_\lam^{-1})$-a.e.~$\mu\in \P_2$.
	
\end{thm}

\begin{proof}
	Let $u$ be as in the assumptions. First, we convince ourselves that $u\in \D(\E^{G,\lam})$. This is achieved in three steps. First, we show 
	$u\in \D(\E^{G,\lam})$ in case we additionally have $\varphi_i\in C_b^1(\RR^d)$, $1\leq i\leq m$. Next, we generalize the statement to bounded $\varphi_i$. In the third step, 
	$u\in \D(\E^{G,\lam})$ is shown for $u$ as in the assumptions.

	 If $u$ is given as in \eqref{eq:u}, \eqref{eq:ubounds} and additionally 
	$\varphi_i\in C_b^1(\RR^d)$, $1\leq i\leq m$, then $u\in C_b^1(\P_2)$ follows from 
	\begin{align*}
			&\frac{\d}{\d\varepsilon}u\big(\mu\circ(\id_{\RR^d}+\varepsilon g)^{-1}\big)\Big|_{\varepsilon=0}\\
			&=\frac{\d}{\d\varepsilon}F\big(\mu(\varphi_1\circ (\id_{\RR^d}+\varepsilon g)),\dots,\mu(\varphi_m\circ (\id_{\RR^d}+\varepsilon g))\big)\Big|_{\varepsilon=0}\\
		&=\sum_{i=1}^m\partial_i F\big(\mu(\varphi_1),\dots,\mu(\varphi_m)\big)\int_{\RR^d}\frac{\d}{\d\varepsilon}\varphi_i (x+\varepsilon g(x))\Big|_{\varepsilon=0}\mu(\d x)\\
		&=\sum_{i=1}^m\partial_i F\big(\mu(\varphi_1),\dots,\mu(\varphi_m)\big)\mu\big(\la \na^\euc\varphi_i,g\ra_\euc \big)\\
		&=\sum_{i=1}^m\partial_i F\big(\mu(\varphi_1),\dots,\mu(\varphi_m)\big)\la \na^\euc\varphi_i,g\ra_{T_\mu},\qquad g\in T_\mu,\,\mu\in\P_2,
	\end{align*}
	i.e.~$Du(\mu)=\sum_{i=1}^m\partial_i F\big(\mu(\varphi_1),\dots,\mu(\varphi_m)\big) \na^\euc\varphi_i$, $\mu\in\P_2$.
	
	Now, we take $\kappa\in C^1(\RR^d)$ with compact support and $\eins_{[-1,1]^d}(x)\leq \kappa(x)\leq 1 $, $|\na^\euc\kappa(x)|_\euc\leq 1$ for $x\in\RR^d$.
	If $u$ is given as in \eqref{eq:u}, \eqref{eq:ubounds} and additionally 
	$\varphi_i$ are bounded functions, $1\leq i\leq m$, then
	\begin{equation*}
		F\Big(\int_{\RR^d}\kappa\big(\tfrac x n\big)\varphi_1(x)\mu(\d x),\dots,\int_{\RR^d}\kappa\big(\tfrac x n\big)\varphi_m(x)\mu(\d x)\Big)=:u_n(\mu)\overset{n\to\infty}{\longrightarrow }u(\mu),
		\qquad \mu\in\P_2,
	\end{equation*}
	and moreover 
	\begin{equation*}
		\|Du_n(\mu)\|_{T_\mu}^2\leq C\max_{1\leq i\leq m}\mu\Big(\big|\na^\euc(\kappa(\tfrac \cdot n)\varphi_i)\big|_\euc^2\Big)
		\leq C\max_{1\leq i\leq m}\mu\Big(\big(\tfrac{1}{n}|\varphi_i|+|\na^\euc\varphi_i|_\euc\big)^2\Big),\quad \mu\in\P_2,
	\end{equation*}
	where $C:=m^2 \max_{1\leq i\leq m}(\sup_{y\in\RR^m}|\partial_i F(y)|^2)$.
	 Therefore, in case $\varphi_i$ are bounded functions,  \sloppy\cite[Lem.~I.2.12]{MR92} yields $u\in\D(\E^{G,\lam})$ and
	\begin{align}\label{eq:inbetween}
		\E^{G,\lam}(u,u)&\leq\liminf_{n\to\infty}\E^{G,\lam}(u_n,u_n)\\&\nonumber\leq \|A^{-1}\|_{T_\lam,\text{Op}}\liminf_{n\to\infty}\int_{\P_2}\|Du_n(\mu)\|_{T_\mu}^2(G\circ\Psi_\lam^{-1})(\d\mu)
		\\&\leq C\|A^{-1}\|_{T_\lam,\text{Op}}\max_{1\leq i\leq m}\int_{\P_2}\mu\big(|\na^\euc\varphi_i|_\euc^2\big)(G\circ\Psi_\lam^{-1})(\d\mu),\nonumber
	\end{align}
	where $ \|A^{-1}\|_{T_\lam,\text{Op}}$ denotes the operator norm of ${A^{-1}}$ and
	  the integral on the right-hand side of \eqref{eq:inbetween} exists due to \eqref{eq:phibounds} and \eqref{eq:Gmom}.
	
	Now, let ${(\tau_n)}_{n\in\NN}\subseteq  C_b^1(\RR)$ be a sequence such that $|\tau_n'(s)|\leq 1$ for $s\in\RR$, $n\in\NN$, and  $\tau_n(s)=s$ for $s\in[-n,n]$.
	For general $u$, given as in \eqref{eq:u}, \eqref{eq:ubounds}, we have
	\begin{equation*}
		F\big(\mu(\tau_n\circ\varphi_1),\dots,\mu(\tau_n\circ\varphi_m)\big)=:\tilde u_n(\mu)\overset{n\to\infty}{\longrightarrow }u(\mu),
		\qquad \mu\in\P_2,
	\end{equation*}
	and moreover by \eqref{eq:inbetween},
	\begin{align*}
		\E^{G,\lam}(\tilde u_n,\tilde u_n)&\leq C\|A^{-1}\|_{T_\lam,\text{Op}}\max_{1\leq i\leq m}\int_{\P_2}\mu\Big(\big|\na^\euc(\tau_n\circ\varphi_i)\big|_\euc^2\Big)(G\circ\Psi_\lam^{-1})(\d\mu)\\
		&\leq C\|A^{-1}\|_{T_\lam,\text{Op}}\max_{1\leq i\leq m}\int_{\P_2}\mu\big(|\na^\euc\varphi_i|_\euc^2\big)(G\circ\Psi_\lam^{-1})(\d\mu),\quad \mu\in\P_2,\,n\in\NN,
	\end{align*}
	for the same constant $C$ as above. Also, $|\tilde u_n(\mu)|\leq \tilde C(1+\mu(|\cdot|_\euc^2))$, $\mu\in\P_2$, for some constant $\tilde C\in(0,\infty)$ independent from $n$,
	by Lipschitz continuity of $F$, the contractive property of $\tau_n$, as well as \eqref{eq:phibounds}.
	Again by  \cite[Lem.~I.2.12]{MR92},
	\begin{equation*}
		u\in\D(\E^{G,\lam}),\qquad \E^{G,\lam}(u,u)\leq C\|A^{-1}\|_{T_\lam,\text{Op}}\max_{1\leq i\leq m}\int_{\P_2}\mu\big(|\na^\euc\varphi_i|_\euc^2\big)(G\circ\Psi_\lam^{-1})(\d\mu).
	\end{equation*}

	Now, having verified that $u\in\D(\E^{G,\lam})$ we address the proof of \eqref{eq:generator}, where $u$ is as in the assumptions of the theorem. Let $v\in C_b^1(\P_2)$. Analogously as in the proof of Lemma \ref{lem:closable}, we have $v\circ\Psi_\lam\in W^{1,2}(T_\lam,G)$ and by Remark \ref{rem:imageForm},
	\begin{equation*}
		\E^{G,\lam}(v,u)={\la v\circ\Psi_\lam,u\circ\Psi_\lam\ra}_{W^{1,2}_{A^{-1}}(T_\lam, G)}.
	\end{equation*}
	In the proof of Proposition \ref{prp:divAB}, we have seen $u\circ\Psi_\lam\in W^{1,2}(T_\lam,G)$.
	Hence,
	\begin{equation*}
		\E^{G,\lam}(v,u)=\int_{T_\lam}{\big\la\na(v\circ\Psi_\lam)(f),A^{-1}\na(u\circ\Psi_\lam)(f)\big\ra}_{T_\lam} G(\d f).
	\end{equation*}
	Now, we  use \eqref{eq:AB}, then Proposition \ref{prp:divAB} and \eqref{eq:nauPsi2} to obtain
\begin{align*}
	&\E^{G,\lam}(v,u)
	=-\int_{T_\lam}(v\circ\Psi_\lam)(f)\Big[\textnormal{div}\big(A^{-1}\na(u\circ\Psi_\lam)\big)(f)-{\big\la \na(u\circ\Psi_\lam)(f),f\big\ra}_{T_\lam}\Big]G(\d f)\\
	&=-\int_{T_\lam}(v\circ\Psi_\lam)(f)\bigg[\Big(\sum_{i=1}^m\partial_iF\big(\lam(\varphi_1\circ f),\dots,\lam(\varphi_m\circ f)\big)\big(a(\varphi_i)(f)-{\la f,\na\varphi_i\circ f\ra}_{T_\lam}\big)\Big)\\
	&\qquad\qquad\qquad\qquad\qquad+\Big(\sum_{i,j=1}^m\partial_i\partial_jF\big(\lam(\varphi_1\circ f),\dots,\lam(\varphi_m\circ f)\big)b(\varphi_i,\varphi_j)(f)\Big)\bigg]G(\d f)\\
	&=-\int_{T_\lam}(v\circ\Psi_\lam)(f)\bigg[\Big(\sum_{i=1}^m\partial_iF\big(\Psi_\lam(f)(\varphi_1),\dots,\Psi_\lam(f)(\varphi_m)\big)\Big(\EE_G[a(\varphi_i)|\Psi_\lam](f)-
	{\la \id_{\RR^d},\na\varphi_i\ra}_{T_{\Psi_\lam(f)}}\Big)\Big)\\
	&\qquad\qquad\qquad\qquad\qquad+\Big(\sum_{i,j=1}^m\partial_i\partial_jF\big(\Psi_\lam(f)(\varphi_1),\dots,\Psi_\lam(f)(\varphi_m)\big)\EE_G[b(\varphi_i,\varphi_j)|\Psi_\lam](f)\Big)\bigg]G(\d f)\\
	&=-\int_{T_\lam}v(\mu)\bigg[\Big(\sum_{i=1}^m\partial_iF\big(\mu(\varphi_1),\dots,\mu(\varphi_m)\big)\Big(\EE_G[a(\varphi_i)|\Psi_\lam=\mu]-
	{\la \id_{\RR^d},\na\varphi_i\ra}_{T_\mu}\Big)\Big)\\
	&\qquad\qquad\qquad\qquad+\Big(\sum_{i,j=1}^m\partial_i\partial_jF\big(\mu(\varphi_1),\dots,\mu(\varphi_m)\big)\EE_G[b(\varphi_i,\varphi_j)|\Psi_\lam=\mu]\Big)\bigg](G\circ\Psi_\lam^{-1})(\d \mu).
\end{align*}
This concludes the proof.
\end{proof}

The subsequent corollary is an application of Theorem \ref{thm:thmL} for some particularly relevant choices of $u$.
For $l\in\{1,\dots,d\}$ we denote the projection $\RR^d\ni x\mapsto x_l\in\mathbb R$ onto the $l$-th coordinate by $\pi_l$.
 For $\mu\in\P_2$, $k,l\in\{1,\dots,d\}$,  we set
\begin{equation}\label{eq:momentfct}
	M^{(1)}_l(\mu):=\mu(\pi_l),\qquad M^{(2)}_{k,l}(\mu):=\mu(\pi_k\cdot \pi_l)\qquad\text{and}\qquad |\mu|_2:={\mu(|\cdot|_\euc^2)}^{\frac 1 2}.
\end{equation}
Moreover, let ${(g_n)}_{n\in\NN}$ be an orthonormal basis in $T_\lam$ and
\begin{equation*}
	C_{k,l}:=\sum_{n=1}^\infty\lam\Big((\pi_k\circ g_n)\big(\pi_l\circ  (A^{-1}g_n)\big)\Big).
\end{equation*}
The constant $C_{l,k}$ does not depend on the choice of orthogonal basis and $\sum_{l=1}^dC_{ll}=\tra(A^{-1})$, the trace of $A^{-1}$.

\begin{cor}\label{cor:normSq}
	$\D(L_{G,\lam})$ contains every function defined in \eqref{eq:momentfct}  and
	\begin{align}\label{eq:normSq}
		L_{G,\lam}M^{(1)}_l(\mu)&=-M^{(1)}_l(\mu),\\ L_{G,\lam}M^{(2)}_{k,l}(\mu)&=-2M^{(2)}_{k,l}(\mu)+2C_{k,l},\nonumber\\
		L_{G,\lam}\big(|\cdot|_2^2\big)(\mu)&=-2|\mu|_2^2+2\tra (A^{-1})\nonumber
	\end{align}
	for $(G\circ\Psi_\lam^{-1})$-a.e.~$\mu\in\P_2$.
\end{cor}
The representatives on the right-hand side of \eqref{eq:normSq} are continuous in $\mu\in\P_2$.

\subsection{\texorpdfstring{$L_{G,\lam}$}{L} as a second order differential operator on test functions}\label{sec:diffOp}

We consider 
 \begin{equation}\label{eq:uphi}
 	\varphi\in C^2(\RR^d),\quad\sup_{x\in\RR^d}|\partial_k\partial_l\varphi(x)|<\infty,\,1\leq k,l\leq d,\qquad u:\P_2\ni\mu\mapsto\mu(\varphi),
 \end{equation}
and show $L_{G,\lam}u(\mu)=\mu(\scr L_{G,\lam}^\mu\varphi)$, $G\circ\Psi_\lam^{-1}\text{-a.e.~}\mu\in\P_2$, with
\begin{equation}\label{eq:Ldiff}
	\scr L_{G,\lam}^\mu \varphi(x)=\sum_{k=1}^d\Big[\Big(\sum_{l=1}^d\rho^\mu_{k,l}(x)\partial_k\partial_l\varphi(x)\Big)-x_k\partial_k\varphi(x)\Big],
	\qquad \mu\text{-a.e.~}x\in\RR^d,
\end{equation}
and coefficients $\rho^\mu_{k,l}\in L^1(\RR^d,\mu)$, depending only on $\mu\in\P_2$, $1\leq k,l\leq d$, and the underlying Gaussian $G$. 
This representation of $L_{G,\lam}$ is a consequence of Theorem \ref{thm:thmL}.
 The coefficients $\rho^\mu_{k,l}$ originate from the Gaussian covariance as explained in the following paragraph.

We denote by ${(g_n)}_{n\in\NN}\subset T_\lam$ an orthonormal eigenbasis of the covariance operator $A^{-1}$ of $G$, i.e.~$A_n^{-1}g_n=\alpha_n  g_n$ for some $\alpha_n\in(0,\infty)$,
and ${(\alpha_n)}_n$ is the sequence of eigenvalues listed with multiplicities. 
Since
\begin{equation*}
\sup_{N\in\NN}\int_{\RR^d}\sum_{n=1}^N \alpha_n|\pi_k\circ g_n||\pi_l\circ g_n|\d\lam\leq	\sum_{n\in\NN}\int_{\RR^d}\alpha_n\la g_n,g_n\ra_\euc^2\d\lam=\tra(A^{-1})
\end{equation*}
for $1\leq k,l\leq d$, we may define
\begin{equation*}
	q_{k,l}:=\sum_{n\in\NN}\alpha_n(\pi_k\circ g_n)(\pi_l\circ g_n)\in L^1(\RR^d,\lambda).
\end{equation*}
\begin{rem}
	The above definition of $q_{k,l}$ does not depend on the choice of an eigenbasis. This follows by writing
	\begin{equation*}
	{[q_{k,l}(x_0)]}_{k,l=1}^d=\lim_{\varepsilon\overset{>}{\to} 0}\sum_{n\in\NN}
	\frac{\alpha_n}{\lam{(B(x_0,\varepsilon))^2}}\int_{B(x_0,\varepsilon)\times B(x_0,\varepsilon)}g_n(x)g_n(y)^\tr\lam(\d x)\lam(\d y),\qquad \lam\text{-a.e.~}x_0\in\RR^d,
	\end{equation*}
	where $B(x_0,\varepsilon)$ denotes the Euclidean $\varepsilon$-ball around a point $x_0\in\RR^d$.
	Clearly, for each $\varepsilon>0$, the integral on the right-hand side is invariant under a change of eigenbasis ${(g_n)}_{n}$ since
	the expression $\sum_{n\in\NN}\alpha_ng_n(x)g_n(y)^\tr$, $x,y\in\RR^d$, coincides with the matrix-valued integral kernel of $A^{-1}$, meaning
	\begin{equation*}
		{\la h_1,A^{-1}h_2\ra}_{T_\lam}=\sum_{n\in\NN}\alpha_n\int_{\RR^d}\int_{\RR^d}\la h_1(x),g_n(x)\ra_\euc \la g_n(y),h_2(y)\ra_\euc\lam(\d x)\lam(\d y),
		\qquad h_1,h_2\in T_\lam.
	\end{equation*}
\end{rem}
We understand $q_{k,l}$ as an element in $L^1(T_\lam\times\RR^d,G\times\lam)$,
\begin{equation*}q_{k,l}:T_\lam\times\RR^d\ni (f,x)\mapsto q_{k,l}(x)\in\RR
\end{equation*}
(only depending on the argument $x$),
  and set
\begin{equation}\label{eq:vecPsilam}
\vec\Psi_\lam(f,x):T_\lam\times\RR^d\ni(f,x)\longmapsto \big(\Psi_\lam(f),f(x)\big)\in\P_2\times\RR^d.
\end{equation}
Now, let $\P_2\times\RR^d\ni (\mu,x)\mapsto \rho^\mu_{k,l}(x)\in\RR$ denote a measurable function, which composed with $\vec\Psi_\lam$ yields the conditional expectation of
$q_{k,l}$ onto the $\sigma$-algebra generated by $\vec\Psi_\lam$ w.r.t.~the measure $G\times\lam$, i.e.
\begin{equation*}
	\rho_{k,l}^{\lam\circ f^{-1}}(f(x)):=\EE_{G\times\lam}\big[q_{k,l}|\vec\Psi_\lam\big](f,x),\qquad \lam\text{-a.e.~}x\in\RR^d,\,G\text{-a.e.~}f\in T_\lam.
\end{equation*} 
That means, $\{\rho^\mu_{k,l}:\mu\in\P_2\}$ are determined by the equation
\begin{multline}\label{eq:rhoKL}
	\int_{\P_2}\int_{\RR^d}h(\mu,x)\rho^\mu_{k,l}(x) \mu(\d x)(G\circ\Psi_\lam^{-1})(\d\mu)\\=
	\int_{T_\lam}\int_{\RR^d} h\big(\vec\Psi_\lam(f,x)\big)q_{k,l}(x)\lam(\d x)G(\d f)\quad\text{for all bounded measurable }h:\P_2\times\RR^d\to\RR,
\end{multline}
or equivalently, $\P_2\times\RR^d\ni (\mu,x)\mapsto \rho^\mu_{k,l}(x)\in\RR$ coincides with the Radon-Nikodym derivative of a finite, signed measure,
\begin{equation}\label{eq:RN}
\rho^\mu_{k,l}(x)=\frac{\d\big((G\times q_{k,l}\lam)\circ\vec\Psi_\lam^{-1}\big)}{\d \big((G\times \lam)\circ\vec\Psi_\lam^{-1}\big)}\,(\mu,x).
\end{equation}

\begin{rem}\label{rem:trace}
	For $\mu\in\P_2$, we set  $\tra(\rho^{\mu}_{\cdot,\cdot}):=\sum_{k=1}^d\rho^\mu_{k,k}\in L^1(\RR^d,\mu)$.
	Then,
	\begin{equation*}
		\mu\big(\tra(\rho^{\mu}_{\cdot,\cdot})\big)=\tra(A^{-1})\qquad\text{for } (G\circ\Psi_\lam^{-1})\text{-a.e.~}\mu\in\P_2,
	\end{equation*}
	as follows from \eqref{eq:rhoKL} and $\sum_{k=1}^d\lam(q_{k,k})=\tra(A^{-1})$.
\end{rem}

\begin{prp}\label{prp:diffO}
	 For $\varphi$, $u$  as in \eqref{eq:uphi} the generator of $(\E^{G,\lam},\D(\E^{G,\lam}))$ is given by $$L_{G,\lam}u(\mu)=\mu(\scr L_{G,\lam}^\mu\varphi),$$ $G\circ\Psi_\lam^{-1}\text{-a.e.~}\mu\in\P_2$,  with second-order differential operator $\scr L_{G,\lam}^\mu$ as in  \eqref{eq:Ldiff}.	
\end{prp}
\begin{proof}
	Let $a(\varphi)$ be as in \eqref{eq:hess}. For bounded, measurable $v:\P_2\to\RR$, we have by \eqref{eq:rhoKL},
	\begin{align*}
		&\int_{\P_2}\EE_G[a(\varphi)|\Psi_\lam=\mu]v(\mu)(G\circ\Psi_\lam^{-1})(\d\mu)=\int_{T_\lam}a(\varphi)(f)v(\Psi_\lam(f))G(\d f)
		\\&=\sum_{k,l=1}^d\int_{T_\lam}\int_{\RR^d}q_{k,l}(x)\partial_k\partial_l\varphi (f(x))v(\Psi_\lam(f))\lam(\d x) G(\d f)\\
		&=\sum_{k,l=1}^d\int_{\P_2}\int_{\RR^d}\rho^\mu_{k,l}(x) \partial_k\partial_l\varphi (x) v(\mu)\mu(\d x)(G\circ\Psi_\lam^{-1})(\d \mu)\\
		&=\int_{\P_2}\Big(\sum_{k,l=1}^d\mu\big(\rho^\mu_{k,l} \partial_k\partial_l\varphi \big) \Big)v(\mu)(G\circ\Psi_\lam^{-1})(\d \mu).
	\end{align*}
	Now, the claim follows from Theorem \ref{thm:thmL}.
\end{proof}

Let us now take into account Condition (G3) from Section \ref{sec:supp}.
The representation formula of $L_{G,\lam}u$ as in Proposition \ref{prp:diffO} remains true for a larger class of functions $\varphi$,
having only polynomial growth bounds for the second order partial derivatives, given that $\rho_{k,l}^\mu(x)$ admit polynomial growth bounds uniformly in $\mu$.

\begin{lem}\label{lem:rhoB}
	Let  $u:\P_2\ni\mu\mapsto\mu(\varphi)$ with $\varphi\in C^2(\RR^d)$.
	We assume that Condition (G3)  holds and for some $p\in[1,\infty)$,
	\begin{equation*}
		\sup_{x\in\RR^d}\frac{|\partial_k\partial_l\varphi(x)|}{1+|x|_\euc^p}<\infty,\qquad \sup_{\mu\in\P_2}\sup_{x\in\RR^d}\frac{\rho_{k,l}^\mu(x)}{1+|x|_\euc^p}<\infty,\qquad1\leq k,l\leq d.
	\end{equation*}
	Then,  $$L_{G,\lam}u(\mu)=\mu(\scr L_{G,\lam}^\mu\varphi),$$ $G\circ\Psi_\lam^{-1}\text{-a.e.~}\mu\in\P_2$,  with second-order differential operator $\scr L_{G,\lam}^\mu$ as in  \eqref{eq:Ldiff}.	
\end{lem}
\begin{proof}
	Let $\kappa\in C^2(\RR^d)$ be compactly supported and $\eins_{[-1,1]^d}(x)\leq \kappa(x)$ for $x\in\RR^d$.
	We set $\kappa_\varepsilon:=\kappa(\varepsilon\,\cdot\,)$, $\varepsilon>0$.
	Condition (G3) via Lebesgue's dominated convergence implies that for any measurable function $\tilde \varphi:\RR^d\to\RR$ with polynomial growth bound,
	\begin{equation*}
		\sup_{x\in\RR^d}\frac{|\tilde\varphi(x)| }{1+|x|_\euc^{\tilde p}}<\infty\qquad\text{for some }\tilde p\in[1,\infty),
	\end{equation*}
	we have the following convergence of functions in $L^2(\P_2,G\circ\Psi_\lam^{-1})$:
	\begin{equation*}
		\mu(\tau_\varepsilon\tilde \varphi)\overset{\varepsilon\to 0}{\longrightarrow}\mu(\tilde\varphi ),\qquad
		\mu((\partial_k\tau_\varepsilon)\tilde \varphi)\overset{\varepsilon\to 0}{\longrightarrow}0,\qquad
		\mu((\partial_k\partial_l\tau_\varepsilon)\tilde \varphi)\overset{\varepsilon\to 0}{\longrightarrow}0,
	\end{equation*}
	$1\leq k,l\leq d.$ Thus, for $u$, $\varphi$ as above and $u_\varepsilon(\mu):=\mu(\tau_\varepsilon\varphi)$ we conclude by Proposition \ref{prp:diffO},
	\begin{equation*}
		u_\varepsilon(\mu)\overset{\varepsilon\to 0}{\longrightarrow}u(\mu),\qquad L^{G,\lam}u_\varepsilon(\mu)\to\mu(\scr L_{G,\lam}^\mu\varphi)
	\end{equation*}
	in $L^2(\P_2,G\circ\Psi_\lam^{-1})$, with second-order differential operator $\scr L_{G,\lam}^\mu$ as in  \eqref{eq:Ldiff}.	This concludes the proof.
\end{proof}
We state an application of Lemma \ref{lem:rhoB}.
\begin{exa}\label{exa:rbf2}
	If $G$ and $\lam$ ae chosen as in Example \ref{exa:rbf}, then ${[\rho^\mu_{k,l}(x)]}_{k,l=1}^d\in \RR^{d\times d}$ equals the identity matrix $I_{d\times d}$ for $\mu$-a.e.~$x\in\RR^d$, $G\circ\Psi_{\lam}^{-1}$-a.e.~$\mu\in\P_2$, i.e. 
	\begin{equation}\label{eq:Kronecker}
		\rho^\mu_{k,l}=\delta_{kl},\qquad 1\leq k,l\leq d.
	\end{equation}
	Indeed, by \eqref{eq:AS},
	\begin{equation*}
		{\la g,A^{-1}h\ra}_{T_\lam}=\int_{\RR^d}\int_{\RR^d}\la g(x),s(x-y)h(y)\ra_\euc\lam(\d x)\lam(\d y),\qquad g,h\in T_\lam,
	\end{equation*}
	with $s(x,y):=e^{-|x-y|_\euc^2}$.
	So, the matrix-valued integral kernel of $A^{-1}$ is given by $s(x-y)I_{d\times d}$, $x,y\in\RR^d$.
	We denote by ${(g_n)}_{n\in\NN}\subset T_\lam$ an orthonormal eigenbasis with $A_n^{-1}g_n=\alpha_n  g_n$,
	and ${(\alpha_n)}_n$ is the sequence of eigenvalues listed with multiplicities. It follows
	\begin{equation*}
		s(x-y)I_{d\times d}=\sum_{n\in\NN}\alpha_ng_n(x)g_n(y)^\tr,\qquad(\lam\times\lam)\text{-a.e.~}(x,y)\in\RR^d\times\RR^d.
	\end{equation*}
	By the definition of ${[q_{k,l}]}_{k,l=1}^d\in \RR^{d\times d}$ above, however, we must have
	\begin{align*}
		{[q_{k,l}(x_0)]}_{k,l=1}^d&=\lim_{\varepsilon\to 0}\sum_{n\in\NN}
		\frac{\alpha_n}{\lam{(B(x_0,\varepsilon))^2}}\int_{B(x_0,\varepsilon)\times B(x_0,\varepsilon)}g_n(x)g_n(y)^\tr\lam(\d x)\lam(\d y)\\&=s(x_0,x_0)I_{d\times d}=I_{d\times d}
	\end{align*}
	for $\lam$-a.e.~$x_0\in\RR^d$.
	Now, \eqref{eq:Kronecker} follows from \eqref{eq:RN}.
	
	Applying Lemma \ref{lem:rhoB}, we conclude that for any multivariate polynomial  $ \theta:\RR^d\to \RR$ and $u(\mu):=\mu(\theta)$ it holds $u\in \D(L_{G,\lam})$ and
	\begin{equation}\label{eq:poly}
		L_{G,\lam}u(\mu)=\mu\big(\Delta \theta-\la \id_{\RR^d},\na \theta\ra_\euc\big),\qquad G\circ\Psi_\lam^{-1}\text{-a.e.}\mu, 
	\end{equation}
	where $\Delta$, $\na$ denote the Laplace operator and the gradient on $\RR^d$. Of course, the same holds if $\theta\in C^2(\RR^d)$ and all second order partial derivatives of $\theta$ are dominated by polynomials.
\end{exa}
%
%
%
%
%
%

\section{The Fukushima decomposition \texorpdfstring{of ${(\bm\mu_t)}_{t\geq 0}$}{}}\label{sec:Fuku}
In this section we analyze the diffusion ${(\bm\mu_t)}_{t\geq 0}$ associated with $(\E^{G,\lam},\D(\E^{G,\lam}))$ for a Gaussian measure $G$ and $\lam\in\P^\ac$ as in Section \ref{sec:B}.
From here on, we simply write $\E:=\E^{G,\lam}$, $L:=L^{G,\lam}$ for its generator and $\Lam:=G\circ\Psi_\lam^{-1}$. 
To obtain the Fukushima decomposition for ${(\bm\mu_t)}_{t\geq 0}$ we first characterize the square-field operator of $\E$ (Section \ref{sec:SF}), which reflects in the quadratic variation of the martingale part. The main result concerns the martingale solutions for the generator $L$ and is stated in Section \ref{sec:ma}.

The form domain $\D(\E)$ is a Hilbert space with inner product
\begin{equation*}
	(u,v)\mapsto \E_1(u,v):=\E(u,v)+\int_{\P_2}uv\d\Lam,\quad u,v\in\D(\E).
\end{equation*}

\subsection{The square-field operator}\label{sec:SF}
In Proposition \ref{prp:squareF} below, we show the existence of a continuous, bilinear map 
\begin{equation}\label{eq:squareF0}
\Gamma:\D(\E)\times\D(\E)\to L^1(\P_2,\Lam)
\end{equation}
such that
\begin{equation}\label{eq:squareF}
	\int_{\scr P_2} v \Gamma(u,u)\d\Lam= 2\E(u v, u)-\E(u^2,v), \quad u,v\in\D(\E)\cap L^\infty(\scr P_2,\Lam)=:\D(\E)_b.
\end{equation}
$\Gamma$ is called the square-field or carré-du-champ operator of $\E$ (see \cite[Chap.~I]{BH91}) and uniquely characterized by \eqref{eq:squareF}.
Moreover, $\Gamma(u,u)$ is non-negative in $\Lam$-a.e.~sense for $u\in\D(\E)$ and $\Gamma(u,u)\d\Lam$ is called the energy measure.

As a prerequisite to Proposition \ref{prp:squareF}, we explain how the covariance of $G$ acts vector fields over $\P_2$.
We identify two measurable functions $\eta,\xi:\P_2\times\RR^d\to\RR^d$  if 
$\int_{\P_2}\int_{\RR^d}|\eta(\mu,\cdot)-\xi(\mu,\cdot)|_\euc\d\mu\Lam(\d\mu)=0$ and write $\int_{\P_2}^\oplus T_\bullet \d\Lam$ for the resulting set of equivalence classes in
\begin{equation*}
	\Big\{\eta:\P_2\times\RR^d\to\RR^d:\eta\text{ is measurable and } \int_{\P_2}\|\eta(\mu,\cdot)\|_{T_\mu}^2\Lam(\d\mu)<\infty\Big\}.
\end{equation*}
Then, $\int_{\P_2}^\oplus T_\bullet \d\Lam$ is a Hilbert space with inner product
\begin{equation*}
	(\eta,\xi)\,\mapsto \,\int_{\P_2}{\la\eta(\mu,\cdot),\xi(\mu,\cdot)\ra}_{T_\mu}^2\Lam(\d\mu),\qquad \eta,\xi\in \int_{\P_2}^\oplus T_\bullet \d\Lam.
\end{equation*}
With $\vec\Psi_\lam$ as in \eqref{eq:vecPsilam}, we have a natural isometry
\begin{align*}
J:\int_{\P_2}^\oplus T_\bullet \d\Lam\;&\rightarrow\; L^2(T_\lam\times \RR^d,G\times \lam;\RR^d)\\
\eta\;&\mapsto\;\eta\circ\vec\Psi_\lam
\end{align*}
due to 
\begin{multline*}
	\int_{T_\lam\times \RR^d}\big|\eta(\Psi_\lam(f),f(x))\big|_\euc^2\lam(\d x)G(\d f)=\int_{T_\lam\times\RR^d}\big|\eta(\lam\circ f^{-1},x)\big|_\euc^2 (\lam\circ f^{-1})(\d x)G(\d f)\\=
	\int_{\P_2}\|\eta(\mu,\cdot)\|^2_{T_\mu}\Lam(\d\mu).
\end{multline*}
With the covariance operator $A^{-1}$ of $G$, we define a strictly positive definite, bounded, symmetric bilinear form on $\int_{\P_2}^\oplus T_\bullet \d\Lam$,
\begin{equation}\label{eq:defQ}
	(\eta,\xi)\,\mapsto\, \int_{T_\lam} \big\la A^{-1}J\eta(f,\cdot ),J\xi(f,\cdot )\big\ra_{T_\lam}G(\d f) ,\qquad \eta,\xi\in \int_{\P_2}^\oplus T_\bullet \d\Lam.
\end{equation}
In \eqref{eq:defQ}, $J\eta(f,\cdot )$ denotes the function $\RR^d\ni x\mapsto J\eta (f,x)=\eta(\Psi_\lam(f),f(x))\in\RR^d$, which is an element of $T_\lam$ for $G$-a.e.~$f\in T_\lam$,
and $J\xi(f,\cdot )$ analogously. The inner product in \eqref{eq:defQ} decomposes into a family of inner products over $\P_2$, giving rise to a family of Hilbert spaces.
\begin{lem}\label{lem:squareF}
	There exists a family of Hilbert spaces ${(H_\mu)}_{\mu\in \P_2}$ such that:
	\begin{itemize}[leftmargin=*]
	\item[]a) $T_\mu\hookrightarrow H_\mu$ is densely, continuously embedded for $\mu\in\P_2$ and the embedding $T_\mu\hookrightarrow H_\mu$ is a Hilbert-Schmidt operator for $\Lam$-a.e.~$\mu\in \P_2$.
	\item[]b)  $\P_2\ni \mu\mapsto{\la\eta(\mu,\cdot),\xi(\mu,\cdot )\ra}_{H_\mu}\in\RR$ is $\Lam$-integrable for $\eta,\xi\in \int_{\P_2}^\oplus T_\bullet \d\Lam$ and 
	\begin{equation}\label{eq:defH}
		\int_{\P_2}{\la\eta(\mu,\cdot),\xi(\mu,\cdot )\ra}_{H_\mu}\Lam(\d\mu)=\int_{T_\lam} \big\la A^{-1}J\eta(f,\cdot ),J\xi(f,\cdot )\big\ra_{T_\lam}G(\d f).
	\end{equation}
	\end{itemize}
\end{lem}
\begin{proof}
	There is a symmetric, bounded linear operator $K$ on $\int_{\P_2}^\oplus T_\bullet \d\Lam$, determined by the equation
	\begin{equation*}
		\int_{\P_2}{\big\la (K\eta)(\mu,\cdot),\xi(\mu,\cdot )\big\ra}_{T_\mu}\Lam(\d\mu)=\int_{T_\lam} \big\la A^{-1}J\eta(f,\cdot ),J\xi(f,\cdot )\big\ra_{T_\lam}G(\d f)
		,\qquad \eta,\xi\in \int_{\P_2}^\oplus T_\bullet \d\Lam.
	\end{equation*}
	Moreover, this equation implies that $K$ is an injective, positive operator.
	For bounded, measurable  $u:\P_2\to\RR$, we set $u\eta(\mu,x):=u(\mu)\eta(\mu,x)$, $u\xi(\mu,x):=u(\mu)\xi(\mu,x)$ and observe
	\begin{align*}
		&\int_{\P_2}{\big\la (K(u\eta))(\mu,\cdot),\xi(\mu,\cdot )\big\ra}_{T_\mu}\Lam(\d\mu)=\int_{T_\lam} \big\la A^{-1}J(u\eta)(f,\cdot ),J\xi(f,\cdot )\big\ra_{T_\lam}G(\d f)\\
		&=\int_{T_\lam} u(\Psi_\lam(f))\big\la  A^{-1}J\eta(f,\cdot ),J\xi(f,\cdot )\big\ra_{T_\lam}G(\d f)
		=\int_{T_\lam} \big\la  A^{-1}J\eta(f,\cdot ),J(u\xi)(f,\cdot )\big\ra_{T_\lam}G(\d f)\\
		&=\int_{\P_2}{\big\la (K\eta)(\mu,\cdot),(u\xi)(\mu,\cdot )\big\ra}_{T_\mu}\Lam(\d\mu)
		=\int_{\P_2}{\big\la u(\mu)(K\eta)(\mu,\cdot),\xi(\mu,\cdot )\big\ra}_{T_\mu}\Lam(\d\mu).
	\end{align*}
	Hence, $K$ commutes with any multiplication operator of the type $\eta \mapsto u\eta$ for $u$ as above, which
    by virtue of \cite[Chap.~II (§1), Thm.~1]{Dixmier}, is equivalent to the existence of a decomposition ${(K_\mu)}_{\mu\in\P_2}$ (unique in $\Lam$-a.e.~sense):
    \begin{itemize}
    	\item[$\bullet$] $K_\mu$ is a bounded linear operator on $T_\mu$ for $\mu\in\P_2$,
    	\item[$\bullet$] $\P_2\times\RR^d\ni(\mu,x)\mapsto \big(K_\mu\eta(\mu,\cdot)\big)(x)\in\RR^d$ is measurable for $\eta\in \int_{\P_2}^\oplus T_\bullet \d\Lam$,
    	\item[$\bullet$] $(K\eta)(\mu,\cdot)=K_\mu\eta(\mu,\cdot)$ in $T_\mu$ for $\Lam$-a.e.~$\mu\in\P_2$.
    \end{itemize}
    Clearly, $K_\mu$ is a non-negative, symmetric operator for $\Lam$-a.e.~$\mu$, inheriting these properties from $K$. The same can be said regarding injectivity, for which we sketch a proof for the reader's convenience:
    
    We choose a countable, point-separating algebra $\scr C$ of bounded measurable functions $\RR^d\to\RR$ and set 
    \begin{equation*}
    	\mathcal C^{\otimes d}:=\big\{ \RR^d\ni x\mapsto \big(\psi_1(x_1),\dots,\psi_d(x_d)\big)\in\RR^d:\psi_i\in \scr C, 1\leq i\leq d\big\}.
    \end{equation*}
    Since $K$ has a dense range $\scr R(K)$, the closure of $\scr R(K)$ in $\int_{\P_2}^\oplus T_\bullet \d\Lam$ contains every function of the family
    $\big\{\P_2\times \RR^d\ni (\mu,x)\mapsto g(x)$: $g\in \scr C^{\otimes d}\big\}.$
    Hence, for $\Lam$-a.e.~$\mu\in\P_2$, the family $\mathcal C^{\otimes d}$ is contained in the $T_\mu$-closure  of $\scr R(K_\mu)$, the range of $K_\mu$. However, the orthogonal complement of $\mathcal C^{\otimes d}$ in $T_\mu$ equals $\{0\}$ for every $\mu\in\P_2$.
    Thus,  $K_\mu$ is injective for $\Lam$-a.e.~$\mu\in\P_2$.

We define $(H_\mu,\|\cdot\|_{H_\mu})$ as the abstract completion of $T_\mu$ w.r.t.~the norm ${\la K_\mu\,\cdot\,,\,\cdot\,\ra}_{T_\mu}^{1/2}$
whenever $K_\mu$ is injective, positive definite (holds for $\Lam$-a.e.~$\mu$). Otherwise, we set $(H_\mu,\|\cdot\|_{H_\mu}):=(T_\mu,\|\cdot\|_{T_\mu})$. Hence, ${(H_\mu)}_{\mu\in \P_2}$ 
is family of Hilbert spaces as claimed in b),
and $T_\mu\hookrightarrow H_\mu$ is densely and continuously embedded for $\mu\in\P_2$. 

To see that the embeddings are Hilbert-Schmidt, we choose a sequence ${(\eta_n)}_{n\in\NN}$ in
$\int_{\P_2}^\oplus T_\bullet \d\Lam$ such that for $\Lam$-a.e.~$\mu\in \P_2$: 
${(\eta_n(\mu,\cdot))}_{n\in \NN\cap[1,\dim(T_\mu)] }$ is an orthonormal basis of $T_\mu$, where $\dim(T_\mu)\in\NN\cup\{\infty\}$ denotes the dimension of $T_\mu$,
and $\eta_n(\mu,\cdot)=0$ for $n\notin \NN\cap[1,\dim(T_\mu)]$. A sequence ${(\eta_n)}_{n\in\NN}$ with that property exists by \cite[Chap.~II (§1), Prop.~1]{Dixmier}.
Then, 
\begin{align*}
	\sum_{n\in\NN}\int_{\P_2}{\|\eta_n(\mu,\cdot)\|}^2_{H_\mu}\Lam(\d\mu)=\sum_{n\in\NN}\int_{T_\lam} \big\la A^{-1}J\eta_n(f,\cdot ),J\eta_n(f,\cdot )\big\ra_{T_\lam}G(\d f)
	\leq \tra(A^{-1}),
\end{align*}
since for $G$-a.e.~$f\in T_\lam$,
\begin{equation*}
	{\big\la J\eta_n(f,\cdot ),J\eta_m(f,\cdot )\big\ra}_{T_\lam}={\big\la \eta_n(\Psi_\lam(f),\cdot ),\eta_m(\Psi_\lam(f),\cdot )\big\ra}_{T_{\Psi_\lam(f)}}=
	\begin{cases}
		1&\text{ if }m=n\leq \dim(T_{\Psi_\lam(f)}),\\
		0&\text{ else.}
	\end{cases}
\end{equation*}
So, $\sum_{n\in\NN} {\|\eta_n(\mu,\cdot)\|}^2_{H_\mu}$ is finite for $\Lam$-a.e.~$\mu\in\P_2$. This concludes the proof.
\end{proof}

The square-field operator $\Gamma(u,u)$ for $u\in C_b^1(\P_2)$ evaluated at $\mu\in\P_2$ coincides with two-times the squared $H_\mu$-norm of the intrinsic derivative of $u$ at $\mu$.

\begin{prp}\label{prp:squareF}
	\begin{itemize}[leftmargin=*]
	\item[]a) $\E$ admits a square-field  operator $\Gamma$ as in  \eqref{eq:squareF0}, \eqref{eq:squareF}.
	\item[]b) $\Gamma(u,u)(\mu)=2{\|Du(\mu)\|}_{  H_\mu}^2$, $u\in C_b^1(\scr P_2)$,  $\Lam$-a.e.~$\mu\in\P_2$.
	\item[]c) If $\varphi$ and $u$ are as in \eqref{eq:uphi}, then
	$\Gamma(u,u)(\mu)=2{\|\na^\euc\varphi\|}_{ H_\mu}^2$,   $\Lam$-a.e.~$\mu\in\P_2$.
	\item[]d) Under Condition (G3) as specified in the beginning of Section \ref{sec:supp}, if $\varphi$ and $u$ are as in Lemma \ref{lem:rhoB}, then
	$\Gamma(u,u)(\mu)=2{\|\na^\euc\varphi\|}_{ H_\mu}^2$,   $\Lam$-a.e.~$\mu\in\P_2$.
	\end{itemize}
\end{prp}

\begin{proof}
	Let $u,v\in C_b^1(\scr P_2)$. Note that $JDu=\na (u\circ\Psi_\lam)$ by \eqref{eq:chainR}. We use \eqref{eq:chainR} to obtain
	\begin{align*}
		&2\E^{G,\lam}(u v, u)-\E^{G,\lam}(u^2,v)\\
		&=2\int_{T_\lam}{\big\la \na\big((uv)\circ\Psi_\lam\big),A^{-1}\na(u\circ\Psi_\lam)\big\ra}_{T_\lam}\d G
		-\int_{T_\lam}{\big\la \na(u^2\circ\Psi_\lam),A^{-1}\na(v\circ\Psi_\lam)\big\ra}_{T_\lam}\d G\\
		&=2\int_{T_\lam}v(\Psi_\lam){\big\la \na(u\circ\Psi_\lam),A^{-1}\na(u\circ\Psi_\lam)\big\ra}_{T_\lam}\d G
		+2\int_{T_\lam}u(\Psi_\lam){\big\la \na(v\circ\Psi_\lam),A^{-1}\na(u\circ\Psi_\lam)\big\ra}_{T_\lam}\d G
		\\&\qquad-2\int_{T_\lam}u(\Psi_\lam){\big\la \na(u\circ\Psi_\lam),A^{-1}\na(v\circ\Psi_\lam)\big\ra}_{T_\lam}\d G\\
		&=2\int_{T_\lam}{\big\la J(vDu)(f,\cdot),A^{-1}JDu(f,\cdot)\big\ra}_{T_\lam}G(\d f)\\&=
		2\int_{\P_2}{\big\la v(\mu) Du(\mu),Du(\mu)\big\ra}_{H_\mu}\Lam(\d\mu)
		=2\int_{\P_2}v(\mu){\big\|  Du(\mu)\big\|}^2_{H_\mu}\Lam(\d\mu),
	\end{align*}
	where Lemma \ref{lem:squareF} provides the equality of the second to last and the last line. The identity
	\begin{equation*}
		2\E^{G,\lam}(u v, u)-\E^{G,\lam}(u^2,v)=2\int_{\P_2}v(\mu){\big\|  Du(\mu)\big\|}^2_{H_\mu}\Lam(\d\mu)
	\end{equation*}
	extends to  $u,v\in\D(\E)_b$ by approximation, since $C_b^1(\scr P_2)$ is dense in $(\D(\E),\E_1^{1/2})$. So, \eqref{eq:squareF} with $\Gamma$ as claimed in b) is shown.
	$\Gamma$ extends uniquely to a continuous, bilinear map $\D(\E)\times\D(\E)\to L^1(\P_2,\Lam)$, see \cite[Chap.~I, Prop.~4.1.3]{BH91}.
	This concludes the proof of a), b).
	
	Before we address the proof of c), we infer that due to \eqref{eq:defH}, \eqref{eq:Gmom} and the estimates of \eqref{eq:phibounds},
	\begin{equation*}
		\int_{\P_2}{\|\na^\euc\varphi\|}_{H_\mu}^2\Lam(\d\mu)\leq \|A^{-1}\|_{T_\lam,\text{Op}}\int_{\P_2}{\|\na^\euc\varphi\|}_{T_\mu}^2\Lam(\d\mu)<\infty
	\end{equation*}
	for $\varphi\in C^2(\RR^d)$, $\sup_{x\in\RR^d}|\partial_k\partial_l\varphi(x)|<\infty$, $1\leq k,l\leq d$.
	
	Next we  show that c) follows from b) by an analogous approximation  as used in the first part of the  proof of Theorem \ref{thm:thmL}.
	Let $\kappa\in C^1(\RR^d)$ have compact support and $\eins_{[-1,1]^d}(x)\leq \kappa(x)\leq 1$, $|\na^\euc\kappa(x)|_\euc\leq 1$ for $x\in\RR^d$.

	If $\varphi$, $u$ are given as in \eqref{eq:uphi} and additionally 
	$\varphi$ is bounded, then
	\begin{align*}
		&\int_{\RR^d}\kappa\big(\tfrac x n\big)\varphi(x)\mu(\d x)=:u_n(\mu)\overset{n\to\infty}{\longrightarrow }u(\mu),\\
		&\Gamma(u_n,u_n)(\mu)=2{\|Du_n(\mu)\|}^2_{H_\mu}= 2{\Big\|\tfrac 1 n\varphi(\cdot)(\na^\euc\kappa)\big(\tfrac \cdot n\big)+\na^\euc\varphi(\cdot)\Big\|}_{H_{\mu}}^2
		\overset{n\to\infty}{\longrightarrow }2{\|\na^\euc\varphi\|}_{H_{\mu}}^2
	\end{align*}
	for $\mu\in\P_2$. Moreover, for $n,m\geq M$ and $M\to\infty$, we have $\int_{\P_2}|u_n-u_m|^2+\Gamma(u_n-u_m,u_n-u_m)\Lam(\d\mu)\to 0$, because
	\begin{align*}
		&\int_{\P_2}|u_n-u_m|^2+\Gamma(u_n-u_m,u_n-u_m)\Lam(\d\mu)=
		\int_{\P_2}\mu\big(\kappa\big(\tfrac \cdot n\big)\varphi(\cdot)-\kappa\big(\tfrac \cdot m\big)\varphi(\cdot)\big)^2\Lam(\d\mu)\\&\qquad+ 2\int_{\P_2}{\Big\|\varphi(\cdot)\Big(\tfrac 1 n(\na^\euc\kappa)\big(\tfrac \cdot n\big)-\tfrac 1 m(\na^\euc\kappa)\big(\tfrac \cdot m\big)\Big)
		+\big(\kappa\big(\tfrac \cdot n\big)-\kappa\big(\tfrac \cdot m\big)\big) \na^\euc\varphi(\cdot)\Big\|}_{H_{\mu}}^2\Lam(\d\mu)
		\\&\leq\int_{\P_2}\mu\big(\eins_{\RR^d\setminus[-M,M]^d}\cdot|\varphi|\big)^2\Lam(\d\mu) 
		+2\|A^{-1}\|_{T_\lam,\text{Op}}\int_{\P_2}\mu\Big(\eins_{\RR^d\setminus[-M,M]^d}\cdot\big(|\varphi|^2+|\na^\euc\varphi|_\euc^2\big)\Big)\Lam(\d\mu), 
	\end{align*}
	where the last inequality holds by choice of $\kappa$ and \eqref{eq:defH}.
	Therefore, in case $\varphi$ is bounded the statement of b) is proven.

	Now, let ${(\tau_n)}_{n\in\NN}\subseteq  C_b^1(\RR)$ be a sequence such that $\sup_{n\in\NN}|\tau_n'(s)|\leq 1$ for $s\in\RR$ and  $\tau_n(s)=s$ for $s\in[-n,n]$.
	 For general $\varphi$, $u$ are given as in \eqref{eq:uphi}, we have
	\begin{align*}
		&\mu(\tau_n\circ\varphi)=:\tilde u_n(\mu)\overset{n\to\infty}{\longrightarrow }u(\mu),\\
&\Gamma(\tilde u_n,\tilde u_n)(\mu)=2{\|(\tau'_n\circ\varphi)\na^\euc\varphi\|}^2_{H_\mu}
\overset{n\to\infty}{\longrightarrow }2{\|\na^\euc\varphi\|}_{H_{\mu}}^2
\end{align*}
	for $\mu\in\P_2$. Moreover, for $n,m\geq M$ and $M\to\infty$, we have $\int_{\P_2}|u_n-u_m|^2+\Gamma(u_n-u_m,u_n-u_m)\Lam(\d\mu)\to 0$, because
	\begin{multline*}
		\int_{\P_2}|u_n-u_m|^2+\Gamma(u_n-u_m,u_n-u_m)\Lam(\d\mu)
		\leq\int_{\P_2}\mu\big((\eins_{\RR\setminus[-M,M]}\circ\varphi)|\varphi|\big)^2\Lam(\d\mu) 
		\\+2\|A^{-1}\|_{T_\lam,\text{Op}}\int_{\P_2}\mu\Big((\eins_{\RR\setminus[-M,M]}\circ\varphi)\big|(\tau_n'\circ\varphi-\tau_m'\circ\varphi)\na^\euc\varphi\big|_\euc^2\Big)\Lam(\d\mu). 
	\end{multline*}
	This concludes the proof of c).
	
	Finally, we address the proof of d). Let $\varphi$ and $u$ be as in Lemma \ref{lem:rhoB}. In the proof of Lemma \ref{lem:rhoB} an approximation by functions ${(u_\varepsilon)}_{\varepsilon>0}$ is constructed, which are of the type in \eqref{eq:uphi}. Since $u_\varepsilon\overset{\varepsilon\to 0}{\to}  u$ w.r.t.~$\E_1^{\frac 1 2}$-norm and $\Gamma:\D(\E)\times \D(\E)\to L^1(\P_2,\Lambda)$ is continuous, the statement follows.
\end{proof}

\begin{rem}\label{rem:loc}
	\begin{itemize}[leftmargin=*]
		\item[]a) We denote by $\eins_{\P_2}$ the $\Lam$-class $\P_2\to\RR$ which is a.e.~constantly $1$. Then,  $\eins_{\P_2}\in\D(\E)$ and $\E(\eins_{\P_2},u)=0$ for all $u\in\D(\E)$.
		By \eqref{eq:squareF} and the continuity of $\Gamma$, we have  $$2\E(u,v)=\int_{\P_2}\Gamma(u,v)\d\Lam,\qquad u,v\in\D(\E).$$.
		\item[]b) For $\tau,\chi\in C_b^1(\RR)$ such that $\tau'(s)\chi'(s)=0$, $s\in\RR$, and $u\in C_b^1(\P_2)$,  by Proposition \ref{prp:squareF} b) and Remark \ref{rem:loc} a),
		\begin{align*}
			\E(\tau\circ u,\chi\circ u)&=\int_{\P_2}{\big\langle D(\tau\circ u)(\mu),D(\chi\circ u)(\mu)\big\rangle}_{H_\mu}\Lam(\d\mu)
			\\&=\int_{\P_2}\tau'(u(\mu))\chi'(u(\mu)){\big\langle Du(\mu),Du(\mu)\big\rangle}_{H_\mu}\Lam(\d\mu)=0.
		\end{align*}
		Hence, $\E$ is local in the sense of \cite[Sect.~I.5]{BH91}. That notion is known as the strong local property. It implies $\E(u,v)=0$ whenever $u,v\in\D(\E)$ and  
		$|u|\Lam$, $|v|\Lam$ are measures with disjoint topological support on $\P_2$.
	\end{itemize}
\end{rem}

\subsection{Martingale solutions for the generator}\label{sec:ma}
First, we recall some basic notions in the context of Potential Theory for Dirichlet forms and Markov processes.
An increasing sequence ${(F_n)}_{n\in\NN}$ of closed sets in $\scr P_2$ is called $\E$-nest if
\begin{equation*}
	\bigcup_{n\in\NN}\Big\{u\in\D(\E):u(\mu)=0\text{ for }\Lam\text{-a.e.~}\mu\in \scr P_2\setminus F_n\Big\}
\end{equation*} 
is dense in $(\D(\E),\E_1^{1/2})$. A subset $\N\subseteq \bigcap_{n}(\P_2\setminus F_n)$ is referred to as $\E$-exceptional. Any $\E$-exceptional set is  a $\Lam$-nullset, but not vice versa.
A statement depending on a reference point $\mu\in\scr P_2$ is said to hold $\E$-quasi-everywhere (\lq q.e.\rq) if valid for all $\mu\in \scr P_2\setminus \N$ up to some $\E$-exceptional set 
$\N\subset \scr P_2$.
The term $\E$-quasi-continuous applies to a function $u:\scr P_2\to\RR$ which restricts to a continuous function $u|_{F_n}\in C(F_n)$ for all $n\in\NN$.

The Dirichlet form $(\E,\D(\E))$ we are investigating is quasi-regular on $\scr P_2$ (see \cite[Def.~IV.3.1]{MR92}). This follows from 
the quasi-regularity criterion in \cite[Thm.~2.1]{RWW24} and the estimate
\begin{align*}
	\E(u,u)&=
\int_{T_\lam}{\la Du(\Psi_\lam f)\circ f,A^{-1}(Du(\Psi_\lam f)\circ f)\ra}_{T_\lam}G(\d f)\\
	&\leq{ \|A^{-1}\|}_{T_\lam,\text{Op}}\int_{T_\lam}{\big\| Du(\Psi_\lam f)\circ f\big\|}_{T_\lam}^2G(\d f)
	\leq{ \|A^{-1}\|}_{T_\lam,\text{Op}}\sup_{\mu\in\scr P_2} {\big\| Du(\mu)\big\|}_{T_\mu}^2
\end{align*}
for $u\in C_b^1(\scr P_2)$ with $ \|A^{-1}\|_{T_\lam,\text{Op}}$ denoting the operator norm of ${A^{-1}}$.

By the method of regularization worked out in \cite[Chap.~VI]{MR92}, there exists an $\E$-nest ${(F_n)}_{n\in\NN}$ of compact sets together with a locally compact topological space $X$ such that:
\begin{itemize}
	\item[$\bullet$] $E:=\bigcup_{n\in\NN}F_n\subset X$ is densely included.
	\item The trace topologies of $\P_2$ and $X$ coincide on every set $F_n$, $n\in\NN$.
	\item  $\B(E)=\{A\subseteq E:A\in\B(X)\}$, where $\B(\cdot)$ denotes the respective Borel $\sigma$-algebra, and hence $\Lam$ can be taken as a probability measure on $X$ with $\Lam(X\setminus E)=0$. 
	\item  Identifying an element $u\in L^2(X,\Lam)$ with the unique  element of $L^2(\P_2,\Lam)$ which coincides with $u$ in $\Lam$-a.e.-sense on $E$,
	$(\E,\D(\E))$ is a regular Dirichlet form on $L^2(X,\Lam)$, as defined in \cite[Chap.~1]{FOT11}.
	\item $X\setminus E$ is an $\E$-exceptional subset of $X$.
\end{itemize}
The last property ensures that the notion of $\E$-exceptional subsets of $E$ is unambiguous, i.e.~not depending on whether we take $\E$ as a Dirichlet form on $\P_2$ or on $X$.
A set $\N\subset \P_2$ is $\E$-exceptional if and only if $\N\cap E$ is $\E$-exceptional, and the analogue is true for a set $\N\subset X$.

Let ${(T_t)}_{t\geq 0}$ denote the semigroup in $L^2(\P_2,\Lam)$ respectively $L^2(X,\Lam)$ corresponding to $(\E,\D(\E))$ and $\eins$ be the $\Lam$-class which a.e.~equals $1$.
Since $\E(\eins,u)=0$ for all $u\in\D(\E)$, 
we have $T_t\eins=\eins$.
In proper association with $\E$ (see \cite[Thm.~IV.~3.5 \& Chap.~VI]{MR92}), there exists a  right process
(a right-continuous, normal and strong Markov process)
  \begin{equation*}
  \mathbf M=\big(\Omega,\F, {( \bm \mu_t)}_{t\geq 0}, {(\theta_t)}_{t\geq 0}, {(P_\mu)}_{\mu\in\scr P_2}\big)
  \end{equation*}
  with state space $\scr P_2$ and shift operator $\theta_t:\Omega\to\Omega$, such that $E$ is $\mathbf M$-invariant and the restricted process $\mathbf M|_E$ is a Hunt process on $X$ (by trivial extension as defined in \cite[Rem.~IV.~3.2.3(i)]{MR92}).
  We may assume that $\mathbf M$ is $\Lam$-tight special standard (see \cite[Def.~IV.~1.13 \& Thm.~IV.~3.5]{MR92}).
A realization of ${( \bm \mu_t)}_{t\geq 0}$ is denoted by ${( \bm \mu_t^\omega)}_{t\geq 0}$ for $\omega\in\Omega$.
$\mathbf M$ is a conservative diffusion (see \cite[Chap.~V]{MR92} and \cite[Chap.~4]{FOT11}) and we may assume
\begin{equation*}
	[0,\infty)\ni t\mapsto\bm\mu_t^\omega\in\scr P_2\quad \text {is continuous,}\quad P_\mu\text{-a.e.~}\omega\in\Omega,
\end{equation*}
 for  every $\mu\in\scr P_2$.
The transition function 
\begin{equation}\label{eq:tranf}
	p_tu(\mu):=\int_\Omega u(\bm \mu_t^\omega) P_\mu(\d\omega),\quad\text{ q.e.}~\mu\in\scr P_2,\,t\geq 0,
\end{equation}
for  $u\in L^2(\P_2,\Lam)$  yields an $\E$-quasi-continuous $\Lam$-version of $T_tu$  
(see \cite[Sect.~IV.~2]{MR92}).
\eqref{eq:tranf} is unambiguous in the sense that choosing two different $\Lam$-versions of $u$, the corresponding representatives of $p_tu$ coincide for q.e.~$\mu\in\P_2$.

The minimum completed admissible filtration of ${( \bm \mu_t)}_{t\geq 0}$ is denoted by $\{\scr F_t\}_{t\geq 0}$.
A finite additive functional (AF) of $\mathbf M$ is a stochastic process $Y:={(Y_t)}_{t\geq 0}$ on $(\Omega,\F)$ with values in $\RR\cup\{\infty\}$ satisfying the following two conditions:
\begin{itemize}
	\item $Y_t$ is $\scr F_t$-measurable with $\{\scr F_t\}_{t\geq 0}$.
	\item There exists $B\in\scr F$ (defining set) and an $\E$-exceptional set $\N$ such that $\theta_t(B)\subseteq B$ for $t>0$,  $P_\mu(B)=1$ for $\mu\in\P_2\setminus\N$,
	 and for  $\omega\in B$:
	 \begin{itemize}
	 	\item $Y_0(\omega)=0$,
	 	\item $Y_t(\omega)\in\RR$, $Y_{t+s}(\omega)=Y_t(\omega)+Y_s(\theta_t\omega)$ for $s,t\geq 0$,
	 	\item ${(Y_t(\omega))}_{t\geq 0}$ is cadlag.
	 \end{itemize}
\end{itemize}
$Y$ is called continuous if in the previous line \lq cadlag\rq\,can be replaced by \lq continuous\rq. 
Two AF's $Y^{(1)}$, $Y^{(2)}$ are called equivalent if they admit a common defining set $B$ such that $Y^{(1)}_t(\omega)=Y^{(2)}_t(\omega)$, $t\geq0$, $\omega\in B$.
The energy of an AF $Y$ is defined as
\begin{equation*}
	e(Y):=\lim_{t\downarrow 0}\frac{1}{2t}\int_{\P_2}\int_{\Omega}Y_t^2\d P_\mu\Lam(\d\mu)\in[0,\infty].
\end{equation*}
As discussed in \cite[Sect.~5.2 (III)]{FOT11},
\begin{equation}\label{eq:caf}
	Y_t:=\Omega\ni \omega\mapsto \int_0 ^t u({\bm \mu^\omega_s})\d s,\qquad t\geq 0,
\end{equation}
 $u\in L^2(\P_2,\Lam)$, yields a continuous, finite additive functional with $e(Y)=0$ and $$\int_\Omega |Y_t|\d P_\mu<\infty,\qquad t\geq 0,$$ for q.e.~$\mu\in\P_2$.
 Up to equivalence, \eqref{eq:caf} does not depend on the choice of a $\Lam$-version of $u$.

Let $\varphi\in C^2(\RR^d)$ with  $\sup_{x\in\RR^d}|\partial_k\partial_l\varphi(x)|<\infty$, $1\leq k,l\leq d$. 
With the notions of Sections \ref{sec:diffOp} \& \ref{sec:SF}, the following are finite additive functionals of $\mathbf M$:
\begin{align*}
	N^{[\varphi]}_t&:=\int_0^t\int_{\RR^d}\sum_{k=1}^d\Big[\Big(\sum_{l=1}^d\rho^{\bm \mu_s}_{k,l}(x)\partial_k\partial_l\varphi(x)\Big)-x_k\partial_k\varphi(x)\Big]{\bm \mu_s}(\d x)\d s,\\
	M^{[\varphi]}_t&:=\int_{\RR^d}\varphi\d \bm \mu_t-\int_{\RR^d}\varphi \d \bm\mu_0-N^{[\varphi]}_t,\\
	Z^{[\varphi]}_t&:=2\int_0^t{\|\na^\euc\varphi\|}_{H_{\bm \mu_s}}^2\d s,\qquad\qquad\qquad t\geq 0.
\end{align*}


The theorem below states that ${(\bm\mu_t)}_{t\geq 0}$ is a martingale solution for the generator $L$ characterized in Section \ref{sec:gen}.

\begin{thm}\label{thm:Fuku} Let $\varphi\in C^2(\RR^d)$,  $\sup_{x\in\RR^d}|\partial_k\partial_l\varphi(x)|<\infty$, $1\leq k,l\leq d$.
	
	 ${(M_t^{[\varphi]})}_{t\geq 0}$
	is a square-integrable $P_\mu$-martingale w.r.t.~${(\F_t)}_{t\geq 0}$ for q.e.~$\mu\in\P_2$, and its quadratic variation is equivalent to ${(Z_t^{[\varphi]})}_{t\geq 0}$.
\end{thm}
\begin{proof}
	Let $u(\mu):=\mu(\varphi)$. Due to Proposition \ref{prp:diffO} it holds
	\begin{equation*}
		\int_{\Omega}N_t^{[\varphi]}\d P_\mu=T_tu(\mu)-u(\mu),\qquad t\geq0, \,\Lam\text{-a.e.~}\mu\in\P_2.
	\end{equation*}
	By \cite[Thm.'s 5.2.2 \& 5.2.4]{FOT11} and the discussion in \cite[Sect.~5.2 (III)]{FOT11}, ${(M_t^{[\varphi]})}_{t\geq 0}$ is an AF, which
	 for q.e.$~\mu$ is a square-integrable $P_\mu$-martingale w.r.t.~${(\F_t)}_{t\geq 0}$.
	 Moreover, by \cite[Thm.~5.2.3]{FOT11} the Revuz measure of the quadratic variation process of ${(M_t^{[\varphi]})}_{t\geq 0}$ coincides with the energy measure 
	 $\Gamma(u,u)\Lam$. Now the statement follows, since the Revuz measure uniquely determines an equivalence class of positive, continuous additive functionals  
	 (see \cite[Thm.'s~5.1.3.~\& 5.1.4]{FOT11})
	 and  $\Gamma(u,u)(\mu)=2{\|\na^\euc\varphi\|}_{H_{ \mu}}^2$, $\Lam$-a.e.~$\mu\in\P_2$ by Prop.~\ref{prp:squareF}.
\end{proof}

\begin{cor}\label{cor:secmomProcess}
	Theorem \ref{thm:Fuku} applies to $\varphi(x):=|x|_\euc^2$, $x\in\RR^d$, in which case, due to 
	Remark \ref{rem:trace}
	we obtain that
	\begin{equation*}
		{|\bm\mu_t|}_2^2-{|\bm\mu_0|}_2^2-2\tra(A^{-1})t+2\int_0^t{|\bm\mu_s|}_2^2\d s,
	\end{equation*}
	is a square-integrable $P_\mu$-martingale w.r.t.~${(\F_t)}_{t\geq 0}$ for q.e.~$\mu\in\P_2$, and its quadratic variation is equivalent to
	\begin{equation*}
		8\int_0^t{\|\id_{\RR^d}\|}_{H_{\bm \mu_s}}^2\d s,\qquad\qquad t\geq 0.
	\end{equation*}
\end{cor}

Theorem \ref{thm:Fuku} and Corollary \ref{cor:secmomProcess} provide the decomposition of $\bm\mu_t(\varphi)-\bm\mu_0(\varphi)$, $t>0$, into its drift and martingale part.
The drift ${(N^{[\varphi]}_t)}_{t\geq 0}$ manifests the Ornstein--Uhlenbeck type nature of ${(\bm\mu_t)}_{t\geq 0}$ (see the discussion in Section \ref{sec:intro}), 
where the coefficients $\rho^{\mu}_{k,l}$ depend purely on the covariance of the underlying Gaussian $G$ (see Section \ref{sec:diffOp}). The martingale part
with quadratic variation ${(Z^{[\varphi]}_t)}_{t\geq 0}$ also inherits the covariance of $G$ (see the construction of the Hilbert spaces $H_\mu$, $\mu\in\P_2$, in Section \ref{sec:SF}).
It is coherent with the natural geometric structure of $\P_2$ in the sense that the tangential space $T_\mu=L^2(\RR^d,\mu;\RR^d)$ at any $\mu\in\P_2$ is densely and continuously embedded into $H_\mu$.
In particular, the family of inner products ${(\la \cdot,\cdot \ra_{H\mu})}_{\mu\in\P_2}$ associated with ${(Z^{[\varphi]}_t)}_{t\geq 0}$ recognizes every direction of the tangential space 
$T_\lam$ at any $\mu\in\P_2$, which makes ${(\bm\mu_t)}_{t\geq 0}$ a non-degenerate diffusion.

In Examples \ref{exa:rbf} \& \ref{exa:rbf2} above a particular choice for $G$ is made to the effect that 
 $\rho^{\mu}_{k,l}=\delta_{kl}$. In this case,
we have the following amendment to Theorem \ref{thm:Fuku}.

\begin{exa}\label{exa:rbf3}
	Let $G$ be as in Examples \ref{exa:rbf} \& \ref{exa:rbf2} and
	 \begin{equation*}
		\mathbf M=\big(\Omega,\F, {( \bm \mu_t)}_{t\geq 0}, {(\theta_t)}_{t\geq 0}, {(P_\mu)}_{\mu\in\scr P_2}\big)
	\end{equation*}
	denote the corresponding diffusion process as above.
	 Due to \eqref{eq:poly},
	\begin{equation*}
		\int_\Omega\int_0 ^t\bm\mu_s\big(\Delta \theta-\la \id_{\RR^d},\na \theta\ra_\euc\big)\d s\d P_\mu=T_tu(\mu)-u(\mu),\qquad t\geq0, \,\Lam\text{-a.e.~}\mu\in\P_2,
	\end{equation*}
	for any multivariate polynomial $\theta:\RR^d\to\RR$ and $u(\mu):=\mu(\theta)$.
	Now, concluding in the same way as in the proof of Theorem \ref{thm:Fuku}, there exists an $\E$-exceptional set $\N$ such that
	\begin{equation*}
		\bm\mu_t(\theta)-\bm\mu_0(\theta)-\int_0 ^t\bm\mu_s\big(\Delta \theta-\la \id_{\RR^d},\na \theta\ra_\euc\big)\d s,\qquad t\geq 0,
	\end{equation*}
	is a square-integrable $P_\mu$-martingale w.r.t.~${(\F_t)}_{t\geq 0}$ for $\mu\in\P_2\setminus\N$ with quadratic variation 
	\begin{equation*}2\int_0^t{\|\na\theta\|}_{H_{\bm \mu_s}}^2\d s,\qquad t\geq 0.
	\end{equation*}
	
	Since the multivariate monomials $x\mapsto \prod_{i=1}^dx_i^{q_i}$ with $q_i\in\NN_0$ form a countable linear basis in the space of polynomials,
	we may choose the exceptional set $\N$ above independent of the test function $\omega$. Moreover, as all evaluation maps $\mu\mapsto \mu(\prod_{i=1}^dx_i^{q_i})$ are members of $\D(\E)$ and thus possess $\E$-quasi-continuous representatives, we can choose $\N$ such that 
$		\P_2\setminus \N\subseteq \bigcap_{p\in[1,\infty)}\P_p$.

\end{exa}

\section*{Funding}
Simon Wittmann extends his sincere thanks to the Research Center for Nonlinear Analysis at  Hong Kong Polytechnic University
for  financial support.

\section*{Acknowledgements}

The authors would like to thank Feng-Yu Wang and Panpan Ren for the stimulating discussions and suggestions, which have been a valuable source of inspiration and feedback to us.

{

}

\end{document}